\documentclass[reqno]{amsart}

\usepackage{amsmath}
\usepackage{hyperref}
\usepackage[normalem]{ulem}
\hypersetup
{
	colorlinks,
	citecolor=blue,
	filecolor=red,
	linkcolor=blue,
	urlcolor=black
}

\usepackage{cancel}
\usepackage{amssymb}
\usepackage{amsthm}
\usepackage{extarrows}
\usepackage{enumitem}
\usepackage[numbers]{natbib}
\usepackage{mathrsfs}
\usepackage{xcolor}
\usepackage
[
a4paper,
vmargin=2.7cm,
hmargin=2.3cm
]{geometry}

\newcommand\R{\mathbb R}

\newcommand\T{\mathbb T}
\newcommand\N{\mathbb N}
\newcommand\E{\mathbb E}
\newcommand\e{\varepsilon}
\newcommand\p{\mathbb P}
\newcommand\W{\mathcal W}
\newcommand\G{\mathcal G}
\newcommand\EE{\mathcal E}

\newcommand\s{\mathcal S}
\newcommand\U{\mathcal U}
\newcommand\HH{\mathcal H}

\newcommand\LL{\mathcal L}
\newcommand\Wlip{W^{1,\infty}}
\newcommand\X{\mathcal X}

\newcommand{\norm}[1]{\left\lVert#1\right\rVert}

\theoremstyle{plain}
\numberwithin{equation}{section}
\newtheorem{Theorem}{Theorem}[section]
\newtheorem{Proposition}{Proposition}[section]
\newtheorem{Lemma}{Lemma}[section]
\newtheorem{Definition}{Definition}[section]
\theoremstyle{definition}
\newtheorem{Remark}{Remark}[section]

\newtheorem{Assumption}{Assumption}

\newcommand{\abs}[1]{\lvert#1\rvert}

\begin{document}

	\title[On a stochastic non-local transport equation]{Global existence and blow-up
		for a stochastic transport equation with non-local velocity}
	
	\author[D. Alonso-Or\'{a}n]{Diego Alonso-Or\'{a}n}
\address{Departamento de An\'{a}lisis Matem\'{a}tico, Universidad de La Laguna, C/ Astrof\'{i}sico Francisco S\'{a}nchez s/n, 38271, Spain}
\email{dalonsoo@ull.edu.es}
	\author[Y. Miao]{Yingting Miao}
	\address{Department of Mathematics, South China University of Technology, Guangzhou, Guangdong 510640, PR China}
	\email{yingtmiao2-c@my.cityu.edu.hk}

	\author[H. Tang]{Hao Tang}
	\address{Department of Mathematics, University of Oslo, P.O. Box 1053\\
	 Blindern, N-0316 Oslo, Norway}
	
	\email{haot@math.uio.no}

	\thanks{D.~Alonso-Or\'{a}n is supported by the Spanish MINECO through Juan de la Cierva fellowship FJC2020-046032-I. The major part of this work was carried out when  D.~Alonso-Or\'{a}n and H.~Tang where supported by the Alexander von Humboldt Foundation.}

	
	\date{\today}
	
	
	
	\begin{abstract}
		In this paper we investigate a non-linear and non-local one dimensional transport equation under random perturbations on the real line. We first establish a local-in-time theory, i.e., existence, uniqueness and blow-up criterion for pathwise solutions in Sobolev spaces $H^{s}$ with $s>3$. Thereafter, we give a complete picture of the long time behavior of the solutions based on the type of noise we consider. On one hand, we identify a family of noises such that blow-up can be prevented with probability $1$, guaranteeing the existence and uniqueness of global solutions almost surely. On the other hand, in the particular linear noise case, we show that singularities occur in finite time with positive probability, and we derive lower bounds of these probabilities. To conclude, we introduce the notion of stability of exiting times and show that one cannot improve the stability of the exiting time and simultaneously improve the continuity of the dependence on initial data.
	\end{abstract}
	
	\maketitle
	
	%
	
	\tableofcontents

	\section{Introduction and main results}
	Consider the following one dimensional non-local transport equation 
	\begin{equation} \label{CCF}
		u_t+(\HH u)u_x=0,
	\end{equation}
	where $\mathcal{H}$ denotes the Hilbert transform.
	This equation first appears in the literature due to its analogy with the Birkhoff-Rott equations describing the evolution of vortex sheets with surface tension \cite{BLM, Morlet}. Moreover, \eqref{CCF} can be also viewed as a toy-model of the two-dimensional surface quasi-geostrophic equation (SQG) which describes the evolution of the potential temperature in a rapidly rotating stratified fluid with uniform potential vorticity \cite{HPGS95,Ped}.  A striking result showing the finite time blow-up of classical solutions to \eqref{CCF} for a generic class of smooth initial data was first obtained by C\'ordoba, C\'ordoba and Fontelos in \cite{Cordoba-etal-2005-Annals} by means of complex analysis techniques. After that breakthrough, equation \eqref{CCF} is known as the CCF equation. Subsequent works have shown finite blow-up avoiding complex analysis approach (cf. \cite{Kiselev-1,Silvestre-Vicol-TAMS}). In particular, in the later Silvestre and Vicol provided four elegant and simple real analysis proofs of the blow-up phenomena.

	In this paper, we are interested in stochastic variants of the CCF equation \eqref{CCF}. Indeed, the introduction of stochasticity into ideal fluid dynamics has received special attention over the past two decades. The inclusion of stochastic noise can be
	a way of representing model uncertainty and turbulence.  For example, in weather forecasting, phenomena as cloud formation
	is to this day poorly understood and the inclusion of stochastic noise has become an essential tool for gaining better understanding about it.  Since the pioneering work of Holm in \cite{Holm-2015-ProcA}, where  a variational approach for introducing noise in equations in a fashion that respects the geometry of the system is developed, the literature regarding the analysis of non-linear stochastic partial differential equations with transport type noise has increase substantially (cf. \cite{Alonso-etal-2019-NODEA, Alonso-Bethencourt-2020-JNS,Crisan-Flandoli-Holm-2018-JNS,Crisan-Lang-2019,Flandoli-Luo-2019} and the references therein). To the best of the authors knowledge, there are very few results regarding the CCF model under random perturbations. Only recently,  by applying an abstract framework for singular stochastic partial differential equations (SPDEs) derived by two of the authors, cf. \cite{Alonso-Rohde-Tang-2021-JNLS}, the local existence, uniqueness and blow-up criterion of pathwise solutions to \eqref{CCF} with  transport noise
	has been addressed in the periodic setting, i.e for $x\in\mathbb{T}=\R/2\pi\mathbb Z$.

	To extend the theory developed in \cite{Alonso-Rohde-Tang-2021-JNLS} to the real line case, i.e., $x\in\R$, and to study the noise effect, in this paper we will consider the following stochastic CCF model
	\begin{equation} \label{SCCF problem}
		\left\{
		\begin{aligned}
			&{\rm d}u+(\HH u)u_x{\rm d}t=h(t,u){\rm d}\W,\quad x\in\R, \ t>0,\\
			&u(\omega,0,x)=u_0(\omega,x),\quad x\in\R,
		\end{aligned} 
		\right.
	\end{equation}
	where $\W$ is a cylindrical Wiener process and $h$ is a non-linear function.  In this work, under some natural assumptions collected in Assumption \ref{Assumption-1}, we obtain the local existence, uniqueness and a blow-up criterion of  pathwise solutions to  \eqref{SCCF problem}. The detailed result is stated in Theorem \ref{Local pathwise solution}. Here we notice that classical probabilistic compactness arguments cannot be directly applied to deal with the whole space $\mathbb{R}$ and new ideas are in order, see Remark \ref{remark:Th1} for a more precise explanation.
	
	It is well-known that noise effect is one of the crucial subjects in the study of SPDEs. Indeed, regularization effects due to noise have been observed for various models and different settings. 
	For example, it is known that the well-posedness of linear stochastic transport equation with noise can be established under weaker hypotheses than its deterministic counterpart and restore uniqueness of solutions, cf. \cite{Fedrizzi-Flandoli-2013-JFA,Flandoli-Gubinelli-Priola-2010-Invention}. Regularization effects caused by the noise on flux for stochastic scalar conservation laws have been studied in \cite{Gess-Souganidis-2017-CPAM}. Extensions of the previous works for the
	stochastic transport and continuity equations to $k$-forms has been recently addressed in \cite{Bethencourt-Takao-2019}. In terms of numerical simulations, the regularization effects of noise can be found in \cite{Kroker-Rohde-2012-ANM}. Moreover, for different fluid models with linear multiplicative noise, we refer to \cite{GlattHoltz-Vicol-2014-AP,Kim-2010-JFA,Tang-2018-SIMA} where noise provides a damping effect on the pathwise behavior of solutions.

	%

	Compared to	the deterministic counterpart of \eqref{SCCF problem}, i.e., \eqref{CCF},
	in this artcicle we focus on the 
	following issues for the problem \eqref{SCCF problem} regarding the noise effect: 
	\begin{itemize}
		\item Noise versus finite time blow-up;
		\item Noise versus dependence on initial data.
	\end{itemize}
	
	\subsection{Noise versus finite time blow-up} 	
	In this direction, we attempt to answer the following two important questions:
	\begin{enumerate}[label={\textbf{(Q-\arabic*)}}]
		\item\label{Q:noise global solution} What kind of noise can prevent blow-up?
		\item\label{Q:noise blow up} If blow-up may occur, what is the corresponding probability?
	\end{enumerate}
	We remark here that most previous results in the literature on regularization by noise are restricted to linear equations or linear noises. For instance, we refer to \cite{Fedrizzi-Flandoli-2013-JFA, Flandoli-Gubinelli-Priola-2010-Invention, FNO-18, Kim-2010-JFA, NO-15-NODEA} for linear transport equations, and to \cite{Flandoli-Gubinelli-Priola-2011-SPTA,GlattHoltz-Vicol-2014-AP,Rockner-Zhu-Zhu-2014-SPTA,Tang-2018-SIMA} for linear noise. Therefore, for non-linear SPDEs, it is very natural to analyze the validity of the regularization effects by non-linear noise. Indeed, searching for nonlinear noise such that blow-up can be prevented is important
because it helps us to understand the regularization mechanisms of noise, and this is the main motivation to study question \ref{Q:noise global solution}.   Actually, even in the case of non-linear equations with linear noise, the noise effects are complicated because there are both, examples in positive direction, i.e., noises can  regularize singularities, and negative direction, i.e., noises cannot regularize singularities. For example, for the stochastic 2D Euler equations, coalescence of vortices disappears  (see \cite{Flandoli-Gubinelli-Priola-2011-SPTA}) but noise cannot prevent the formation of shocks in the Burgers' equation (see \cite{Alonso-etal-2019-NODEA,Flandoli-2011-book}). 
	
	For simplicity, we set \ref{Q:noise global solution} in the framework where $h(t,u)\,{\rm d}\W
	=\alpha(t,u)\, {\rm d}W$, with $W$ a standard 1-D Brownian motion and $\alpha$ a non-linear function. We then focus on the system
	\begin{equation}\label{SCCF non blow up Eq}
		\left\{\begin{aligned}
			&{\rm d}u+(\HH u)u_x\, {\rm d}t
			= \alpha(t,u)\, {\rm d}W,\quad x\in\R, \ t>0,\\
			&u(\omega,0,x)=u_{0}(\omega,x), \qquad  x\in \R.
		\end{aligned} \right.
	\end{equation}
	We will show in Theorem \ref{Non blowup} that if $\alpha(t,\cdot)$ grows fast enough, then global existence of pathwise solutions holds true with probability $1$. This is strongly in contrast with its deterministic counterpart where the breakdown of classical solutions to \eqref{CCF} with generic smooth initial data occurs, cf. \cite{Cordoba-etal-2005-Annals}. Hence we justify the idea that fast growing  non-linear noise (\textit{strong} noise) has regularization effects on the solutions in terms of preventing singularities.
	
	 	By Theorem \ref{Non blowup}, we have identified a family of noises that can prevent blow-up, and this partially answers \ref{Q:noise global solution}.
	 Next, we
	will pay our attention to the  case that blow-up may occur. Indeed, as a toy-model for the 2D surface quasi-geostrophic equation (SQG)(and hence for the 3D incompressible Euler equation), analyzing the possible blow-up of solutions is 
	is one of the central questions in the study of  
	non-local transport type equations, cf. \cite{Cordoba-etal-2005-Annals,Dong-2008-JFA, Kiselev-1,Li-Rodrigo-2008,Silvestre-Vicol-TAMS}. In this paper we are also interested in identifying the possible formation of singularities in finite time and estimating its probability, hence given a partial answer to \ref{Q:noise blow up}. 
	Since Theorem \ref{Non blowup} shows that fast enough growing noises can prevent blow-up, it is natural to ponder that singularities can only occur when the noise is somehow weak. Indeed, in contrast to the fast growing noise, we will show that in the case of linear noise (\textit{weak} noise) given by equation
	\begin{equation} \label{linear:SCCF:problem}
		\left\{
		\begin{aligned}
			&{\rm d}u+(\HH u)u_x{\rm d}t=b(t)u\, {\rm d}W,\quad x\in\R, \ t>0,\\
			&u(\omega,0,x)=u_0(\omega,x),\quad x\in\R,
		\end{aligned} 
		\right.
	\end{equation}
	where $b$ is some continuous function and $W$ is a standard 1-D Brownian motion, finite time blow-up cannot be prevented and finite time singularities occur. The precise statement of this result is given in Theorem \ref{linear:theorem:blowup}.  
	
	\subsection{Noise versus dependence on initial data}
	
	Now we turn to the problem of noise effect on the initial-data dependence. There are very few results concerning the noise effect in the direction of dependence on initial data. In this work we will partially answer the following question:
	\begin{enumerate}[label={\textbf{(Q-3)}}]
		\item\label{Q:noise dependence} Whether and how noise can affect the initial-data dependence?
	\end{enumerate}
	The main motivation to consider question \ref{Q:noise dependence} relies on the following observation. 
	On the one hand, regularization provided by noise may look related to regularization effects induced by an additional dissipative term (a Laplacian). On the other hand, if one would add a real Laplacian to the governing equations, parabolic techniques may be used to improve the continuity of the initial-data dependence. For example, in the deterministic incompressible Euler equations, the solution map $u_0\mapsto u$ cannot be better than continuous \cite{Himonas-Misiolek-2010-CMP} but for the deterministic incompressible Navier-Stokes equations with sufficiently large viscosity, it is at least Lipschitz continuous in sufficiently high Sobolev spaces (see pp. 79--81 in \cite{Henry-1981-book}). In the deterministic setting, similar questions regarding the continuity map on the initial-data have been widely investigated for various non-linear dispersive and integrable equations of which we only mention a few related results. Koch and Tzvetkov \cite{Koch-Tzvetkov-2005-IMRN} proved that the solution map of the {B}enjamin--{O}no equation cannot be uniform continuous. For Camassa--Holm type equations, we refer to \cite{Himonas-Kenig-2009-DIE,Himonas-Kenig-Misiolek-2010-CPDE}  for the non-uniform dependence on initial data in Sobolev spaces $H^s$ with $s>3/2$. Similar results in Besov spaces $B^{s}_{p,q}$ first appear in \cite{Tang-Zhao-Liu-2014-AA,Tang-Shi-Liu-2015-MM}, where the critical index $s$ can be also covered. In the case of the SQG system, we refer the reader to \cite{Inci-2018-JDE}. In the stochastic setting, the interplay between regularization provided by noise and the dependence on initial conditions is first studied in \cite{Rohde-Tang-2020-JDDE,Tang-2020-arxiv}. 
	
	In this article we consider question \ref{Q:noise dependence} for \eqref{SCCF problem}. More precisely, we first recall the concept of stability of the exiting time as in \cite{Rohde-Tang-2020-JDDE,Tang-2020-arxiv}. Roughly speaking, this notion refers to the continuous changes of the point in time with respect to the initial condition, where such point is defined as the time when the solution leaves a certain range, see Definition \ref{Definition stability of exiting time} below. Later on, in Theorem \ref{Weak instability}, we show that when $h(t,u)$ satisfies certain conditions (see Assumption \ref{Assumption-2}), the multiplicative noise cannot improve the stability of the exiting time, and, at the same time, improve the continuity of the map $u_0\mapsto u$ defined by \eqref{SCCF problem}.

	\subsection{Notations, definitions and hypotheses}

	We now introduce some notations. 
	$L^2(\R)$ is the usual space of square--integrable functions  on $\R$. For $s\in\R$,  $D^s=(1-\partial_{xx}^2)^{s/2}$ is defined by
	$\widehat{D^sf}(\xi)=(1+\xi^2)^{s/2}\widehat{f}(\xi)$, where $\widehat{g}$ is the Fourier transform of $g$. The Sobolev space $H^s(\R)$ is defined as
	\begin{align*}
		H^s(\R)\triangleq\left\{f\in L^2(\R):\|f\|_{H^s(\R)}^2=\int_\R(1+\xi^2)^s|\widehat{f}(\xi)|^2<+\infty\right\},
	\end{align*}
	in which the inner product $(f,g)_{H^s}$ is given by
	$$(f,g)_{H^s}\triangleq\int_\R(1+\xi^2)^s\widehat{f}(\xi)\cdot\overline{\widehat{g}}(\xi)=(D^sf,D^sg)_{L^2}.$$
	When the function space refers to  $\R$, we will drop $\R$ if there is no ambiguity. $x \lesssim y$ ($x \gtrsim y$) means that
	$x\le c y$ ($x\ge c y$) holds for some universal \textit{deterministic} constant $c$. Such constant may differ from line to line.
	For linear operators $A$ and $B$, the commutator $[A,B]$ is defined by $[A,B]=AB-BA$.

	The triplet $(\Omega, \mathcal{F},\p)$ denotes a complete probability space, where $\p$ is a probability measure on $\Omega$ and $\mathcal{F}$ is a $\sigma$-algebra. $\E X$ is the mathematical expectation
	of $X$ with respect to $\p$. Let
	$\W(t)=\W(\omega,t),\omega\in\Omega$ be a cylindrical Wiener process.  More precisely, we consider a separable Hilbert space $\U$ and let $\{e_k\}$  be a complete orthonormal basis of $\U$. 
	Then we define
	\begin{equation*}
		\W\triangleq\sum_{k=1}^\infty W_ke_k,
	\end{equation*}
	where  $\{W_k\}_{k\geq1}$ is a sequence of mutually independent standard one-dimensional Brownian motions. 
	We call
	$\s=(\Omega, \mathcal{F},\p,\{\mathcal{F}_t\}_{t\geq0}, \W)$ a stochastic basis, where $\{\mathcal{F}_t\}_{t\geq0}$ is a right-continuous filtration endowed on $(\Omega, \mathcal{F},\p)$ such that $\{\mathcal{F}_0\}$
	contains all the $\p$-negligible subsets. 
	
	$\LL_2(\U; \X)$ stands for the set of Hilbert-Schmidt operators from $\U$ to another Hilbert space $\X$.  	For a predictable process $Z\in \LL_2(\U; \X)$,
	\begin{equation*}
		\int_0^t Z{\rm d}\W\triangleq\sum_{k=1}^\infty\int_0^t Z e_k{\rm d}W_k
	\end{equation*}
	is a well-defined $\X$-valued continuous
	square integrable martingale, see \cite{Prato-Zabczyk-2014-Cambridge,Prevot-Rockner-2007-book} for more details. 
	In the sequel of the paper, when a stopping time is defined, we set $\inf\emptyset\triangleq\infty$ by convention.
	%
	%

	We now  give the precise notion of a pathwise solution to \eqref{SCCF problem}.
	
	\begin{Definition}[Pathwise solutions]\label{pathwise solution definition}
		Let $\s=(\Omega, \mathcal{F},\p,\{\mathcal{F}_t\}_{t\geq0}, \W)$ be a fixed stochastic basis. Let $s>3/2$ and $u_0$ be an $H^s$-valued $\mathcal{F}_0$-measurable random variable.
		\begin{enumerate}
			\item A local pathwise solution to \eqref{SCCF problem} is a pair $(u,\tau)$, where $\tau$ is a stopping time satisfying $\p\{\tau>0\}=1$ and
			$u:\Omega\times[0,\infty]\rightarrow H^s$  is an $\mathcal{F}_t$-predictable $H^s$-valued process satisfying
			\begin{equation*}
				u(\cdot\wedge \tau)\in C([0,\infty);H^s)\ \ \p-a.s.,
			\end{equation*}
			and for all $t>0$,
			\begin{equation*} 
				u(t\wedge \tau)-u(0)+\int_0^{t\wedge \tau}
				(\HH u)u_x{\rm d}t'
				=\int_0^{t\wedge \tau}h(t',u){\rm d}\W\ \ \p-a.s.
			\end{equation*}
			\item The local pathwise solutions are said to be pathwise unique, if given any two pairs of local pathwise solutions $(u_1,\tau_1)$ and $(u_2,\tau_2)$ with $\p\left\{u_1(0)=u_2(0)\right\}=1,$ we have
			\begin{equation*}
				\p\left\{u_1(t,x)=u_2(t,x),\ \forall\ (t,x)\in[0,\tau_1\wedge\tau_2]\times \R\right\}=1.
			\end{equation*}
			\item Additionally, $(u,\tau^*)$ is called a maximal pathwise solution to \eqref{SCCF problem} if $\tau^*>0$ almost surely and if there is an increasing sequence $\tau_n\rightarrow\tau^*$ such that for any $n\in\N$, $(u,\tau_n)$ is a pathwise solution to \eqref{SCCF problem} and on the set $\{\tau^*<\infty\}$,
			\begin{equation*} 
				\sup_{t\in[0,\tau_n]}\|u\|_{H^s}\geq n.
			\end{equation*}
			\item If $(u,\tau^*)$ is a maximal pathwise solution and
			$\tau^*=\infty$ almost surely, then we say that the pathwise solution exists globally.
		\end{enumerate}
	\end{Definition}
Inspired by \cite{Rohde-Tang-2020-JDDE,Tang-2020-arxiv}, we introduce the concept on 
	stability of  exiting time in Sobolev spaces. Exiting time, as the name suggests, describes the first time that the solution leaves a given range. More precisely,
	
	\begin{Definition}[Stability of exiting time]\label{Definition stability of exiting time}
		Let $s>3/2$ and $\s=(\Omega, \mathcal{F},\{\mathcal{F}_t\}_{t\geq0},\p, \W)$ be a fixed stochastic basis. Let $u_0\in L^2(\Omega;H^s)$ and $\{u_{0,n}\}\subset L^2(\Omega;H^s)$ be $\mathcal{F}_0$-measurable. For each $n$, let $u$ and $u_n$ be the unique solutions to \eqref{SCCF problem}, as in Definition \ref{pathwise solution definition}, with initial values $u_0$ and $u_{0,n}$, respectively. For any $R>0$, define the $R$-exiting times as
		\begin{equation}
			\tau_n^R\triangleq\inf\{t\geq 0:\|u_n\|_{H^s}>R\},\ \ \tau^R\triangleq\inf\{t\geq 0:\|u\|_{H^s}>R\}.
		\end{equation}
		Then we define the following properties on stability:
		\begin{enumerate}
			\item If $u_{0,n}\rightarrow u_0$ in $H^{s}$ $\p-a.s.$ implies that
			\begin{align}
				\lim_{n\rightarrow\infty}\tau^R_{n}=\tau^R\ \  \p-a.s.,\label{hitting time convergence}
			\end{align}
			then the $R$-exiting time of $u$ is said to be stable.
			\item If $u_{0,n}\rightarrow u_0$ in $H^{s'}$ for all $s'<s$ almost surely,  implies that \eqref{hitting time convergence} holds true,
			the $R$-exiting time of $u$ is said to be strongly stable.
		\end{enumerate}
	\end{Definition}
	
	To study the existence of  pathwise solutions to \eqref{SCCF problem}, we  need the following assumptions on the $h$:
	
	\begin{Assumption}\label{Assumption-1} 
		
		We assume that when $s>3/2$,   $h:[0,\infty)\times H^s\ni (t,u)\mapsto h(t,u)\in \LL_2(\U; H^s)$ is continuous. Furthermore, we assume that there are two non-decreasing  locally bounded functions $f,q:[0,+\infty)\rightarrow[0,+\infty)$ such that
		\begin{itemize}

			\item For any $t>0$ and $u\in H^s$,
			\begin{align}\label{assumption 1 for h}
				\|h(t,u)\|_{\LL_2(\U; H^s)}\leq f(\|u_x\|_{L^{\infty}}+\|\HH u_x\|_{L^\infty})(1+\|u\|_{H^s}).
			\end{align}
			\item 
			For any $t>0$,
			\begin{equation}\label{assumption 2 for h}
				\sup_{\|u\|_{H^s},\|v\|_{H^s}\le N}\left\{{\bf 1}_{\{u\ne v\}}  \frac{\|h(t,u)-h(t,v)\|_{\LL_2(\U; H^s)}}{\|u-v\|_{H^s}}\right\} \le q(N),\ \ N\ge 1.
			\end{equation}
			
		\end{itemize}
	\end{Assumption}

\textbf{Example.} Now we give an example of noise coefficient satisfying Assumption \ref{Assumption-1}. For simplicity, we only consider the 1-D case, i.e., $\W=W$, and $h:(t,u)\mapsto  h(t,u)$ is a map from $[0,\infty)\times H^s$ to $ H^{s}$.  It is easy to verify that if that $h$ can be constructed, the extension of this example to the general  $\LL_2(\U;H^s)$-valued $h$ can be carried out by considering $h(e_k)$, where $\{e_k\}_{k\ge1}$ is a complete orthonormal basis of $\U$ (as 1-D case with suitable coefficient) such that $\sum_{k\geq1}\|h(e_k)\|^2_{H^s}<\infty$. To that purpose, let $G=\frac{1}{2} {\rm e}^{-|x|}$. Then in 1-D case, $(1-\partial_{xx}^2)^{-1}$ can be understood as
	\begin{align}\label{Helmboltz operator 2}
		\left[(1-\partial_{xx}^2)^{-1}f\right](x)=[G\star f](x) \ \text{for}\ f\in L^2(\R),
	\end{align}
	where $\star$ stands for the convolution. Now we let 
	$$h(t,u)=q(t)(1-\partial_{xx}^2)^{-1}\partial_{x}[(u_x)^k+(\HH u_x)^n],\ \ k,n\ge1.$$
	If $q(\cdot)$ is smooth with both upper and lower bounds, then it is easy to see that $h$ satisfies Assumption \ref{Assumption-1}.

	To find global existence, we need some stronger condition on the noise coefficient $h$ and we make the following assumption:
	
	\begin{Assumption}\label{Assumption-alpha}
		We assume that when $s>\frac32$,  $\alpha:[0,\infty)\times H^s\ni (t,u)\mapsto \alpha(t,u)\in H^s$  is continuous. Moreover, we assume the following properties hold true:
		
		\begin{itemize}
			\item $\alpha(\cdot,u)$ is  bounded for all $u\in H^s$ and  there is a non-decreasing  locally bounded function $l(\cdot):[0,\infty)\rightarrow[0,\infty)$ such that
			for any $t\ge 0$,
			\begin{equation}
				\sup_{\|u\|_{H^s},\|v\|_{H^s}\le N}\left\{{\bf 1}_{\{u\ne v\}}  \frac{\|\alpha(t,u)-\alpha(t,v)\|_{H^s}}{\|u-v\|_{H^s}}\right\} \le l(N),\ \ N\ge 1,\ s>3/2.\label{alpha-Lip}
			\end{equation}

			\item Define
			\begin{equation*}
				\mathfrak{G}=\left\{\G\in C^2([0,\infty);[0,\infty)): \G(0)=0,\ \G'(x)>0, \ \G''(x)\le 0 \ \text{and} \ \lim_{x\to\infty} \G(x)=\infty\right\},
			\end{equation*}
			and we assume that there is a function $\G\in \mathfrak{G}$ and constants $K_1, K_2>0$ such that for all $(t,u)\in [0,\infty)\times H^s$ with $s>5/2$,
		\begin{align} 
			\G'(\|u\|^2_{H^{s-1}})\mathcal{M}(t)+
			2\G''(\|u\|^2_{H^{s-1}})\left|\left(\alpha(t,u), u\right)_{H^{s-1}}\right|^2 
			\leq K_1
			-K_2 \frac{\left\{\G'(\|u\|^2_{H^{s-1}})\left|\left(\alpha(t,u), u\right)_{H^{s-1}}\right|\right\}^2}{ 1+\G(\|u\|^2_{H^{s-1}})},\label{noise alpha Lyapunov}
		\end{align}
			where 
			\begin{eqnarray*}
			\mathcal{M}(t)=2Q(\|u_x\|_{L^\infty}+\|\HH u_x\|_{L^\infty})\|u\|^2_{H^{s-1}}
				+\|\alpha(t,u)\|^2_{H^{s-1}}
			\end{eqnarray*}
		and	$Q$ is the constant given in Lemma \ref{Huux Hs inner product}.
			
		\end{itemize}
	\end{Assumption}

\textbf{Example.}
Let $q(t)$ be a continuous function such that $q_*<q^2(t)<q^*$ for all $t$ and let $Q$ be the constant given in Lemma \ref{Huux Hs inner product} below. Then it is easy to check that 
\begin{align}
\alpha(t,u)\triangleq q(t)\left(1+\|u_x\|_{L^\infty}+\|\HH u_x\|_{L^\infty}\right)^\theta u \label{example:global:alpha}
\end{align}
with 
\begin{equation}\label{example:global:coefficients}
	\text{either}\ \theta>\frac12,\ q^*>q_*>0\ \text{or}\ \theta=\frac12,\ q^*>q_*>2Q
\end{equation}
satisfies Assumption \ref{Assumption-alpha} with $\G(x)=\log(1+x)\in\mathfrak{G}$. For simplicity we only prove that \eqref{noise alpha Lyapunov}. Indeed, we observe that
\begin{align*}
&\G'(\|u\|^2_{H^{s-1}})\mathcal{M}(t)+
2\G''(\|u\|^2_{H^{s-1}})\left|\left(\alpha(t,u), u\right)_{H^{s-1}}\right|^2+K_2 \frac{\left\{\G'(\|u\|^2_{H^{s-1}})\left|\left(\alpha(t,u), u\right)_{H^{s-1}}\right|\right\}^2}{ 1+\G(\|u\|^2_{H^{s-1}})}\notag\\
=\, & \frac{2Q(\|u_x\|_{L^\infty}+\|\HH u_x\|_{L^\infty})\|u\|^2_{H^{s-1}}
	+q^2(t)\left(1+\|u_x\|_{L^\infty}+\|\HH u_x\|_{L^\infty}\right)^{2\theta} \|u\|^2_{H^{s-1}}}{1+\|u\|^2_{H^{s-1}}}\notag\\
&-\frac{2q^2(t)\left(1+\|u_x\|_{L^\infty}+\|\HH u_x\|_{L^\infty}\right)^{2\theta} \|u\|^4_{H^{s-1}}}{\left(1+\|u\|^2_{H^{s-1}}\right)^2}\notag +K_2\frac{q^2(t)\left(1+\|u_x\|_{L^\infty}+\|\HH u_x\|_{L^\infty}\right)^{2\theta} \|u\|^4_{H^{s-1}}}{\left(1+\|u\|^2_{H^{s-1}}\right)^2\left(1+\log(1+\|u\|^2_{H^{s-1}})\right)}\notag\\
\leq\, &2Q(\|u_x\|_{L^\infty}+\|\HH u_x\|_{L^\infty})+q^2(t)\left(1+\|u_x\|_{L^\infty}+\|\HH u_x\|_{L^\infty}\right)^{2\theta}\notag\\
&-\frac{2q^2(t)\left(1+\|u_x\|_{L^\infty}+\|\HH u_x\|_{L^\infty}\right)^{2\theta} \|u\|^4_{H^{s-1}}}{\left(1+\|u\|^2_{H^{s-1}}\right)^2}+K_2\frac{q^2(t)\left(1+\|u_x\|_{L^\infty}+\|\HH u_x\|_{L^\infty}\right)^{2\theta} }{\left(1+\log(1+\|u\|^2_{H^{s-1}})\right)}:=\mathfrak{J}
\end{align*}
If $\|u_x\|_{L^\infty}+\|\HH u_x\|_{L^\infty}$ is bounded, then $\mathfrak{J}$ is also bounded, and hence can be controlled by some constant $K_1>0$. To prove \eqref{noise alpha Lyapunov}, we only need to check that $\mathfrak{J}$ can be also controlled by $K_1$ when $\|u_x\|_{L^\infty}+\|\HH u_x\|_{L^\infty}\rightarrow\infty$.
Let $\mathcal{P}=q^2(t)\left(1+\|u_x\|_{L^\infty}+\|\HH u_x\|_{L^\infty}\right)^{2\theta}$.	Due to the embedding $H^{s-1}\hookrightarrow W^{1,\infty}$, when $\|u_x\|_{L^\infty}+\|\HH u_x\|_{L^\infty}\rightarrow\infty$, $\|u\|_{H^{s-1}}\rightarrow\infty$. 
Hence, we have that
	\begin{align*}
		\limsup_{\|u_x\|_{L^\infty}+\|\HH u_x\|_{L^\infty}\rightarrow +\infty}\mathfrak{J}
		&\leq \,\limsup_{\|u_x\|_{L^\infty}+\|\HH u_x\|_{L^\infty}\rightarrow +\infty}	\left\{\frac{2Q(\|u_x\|_{L^\infty}+\|\HH u_x\|_{L^\infty})}{\mathcal{P}}
		+	1-\frac{2\|u\|^4_{H^{s-1}}}{\left(1+\|u\|^2_{H^{s-1}}\right)^2}\right.\\
		&\hspace*{5cm}\left.
		+K_2\frac{\|u\|^2_{H^{s-1}}}{\left(1+\|u\|^2_{H^{s-1}}\right)^2\left(1+\log(1+\|u\|^2_{H^{s-1}})\right)}\right\}\mathcal{P}.
	\end{align*}
Moreover, if \eqref{example:global:coefficients} is satisfied,  $$\limsup_{\|u_x\|_{L^\infty}+\|\HH u_x\|_{L^\infty}\rightarrow +\infty}\frac{2Q(\|u_x\|_{L^\infty}+\|\HH u_x\|_{L^\infty})}{\mathcal{P}}<1$$ and consequently
$\limsup_{\|u_x\|_{L^\infty}+\|\HH u_x\|_{L^\infty}\rightarrow +\infty}\mathfrak{J}\leq0$. 
Therefore $\alpha$ given by \eqref{example:global:alpha} and $\G(\cdot)=\log(1+\cdot)$ satisfies \eqref{noise alpha Lyapunov}.

In the linear noise case \eqref{linear:SCCF:problem}, we make the following assumption on $b$:
	
	\begin{Assumption}\label{Assumption-3} 
	We assume $b(t)$ in \eqref{linear:SCCF:problem} satisfies that $b(t)\in C\left([0,\infty);[0,\infty)\right)$ and there exists some $b^*>0$ such that $b^2(t)<b^*$ for all $t\geq 0$.
	\end{Assumption}

	Finally, we assume the following conditions to study the question \ref{Q:noise dependence}:
	\begin{Assumption}\label{Assumption-2} 
		When considering \eqref{SCCF problem} in Section \ref{sect:weak instability}, we assume that for $s>3/2$, $h:[0,\infty)\times H^s\ni (t,u)\mapsto h(t,u)\in \LL_2(\U; H^s)$ is continuous. Moreover, we assume the following:
		
		\begin{itemize} 
			
			\item There exists a non-decreasing and locally bounded functions $l(\cdot):[0,+\infty)\rightarrow[0,+\infty)$ such that for any $t\ge 0$ and $u\in H^s$ with $s>3$, we have that
			\begin{align*}
				\|h(t,u)\|_{\LL_2(\U; H^s)}\leq l(\|u_x\|_{L^{\infty}}+\|\HH u_x\|_{L^{\infty}})\|u\|_{H^s}.
			\end{align*} 
			\item  There exists a non-decreasing and locally bounded functions $g(\cdot):[0,+\infty)\rightarrow[0,+\infty)$ and a real number $\sigma_0\in(3/2,7/4)$ such that for all $N\ge1$,
			\begin{align*}
				\sup_{t\ge0,\|u\|_{H^s}\le N}\|h(t,u)\|_{\LL_2(\U; H^{\sigma_0})}\leq g(N) {\rm e}^{-\frac{1}{\|u\|_{H^{\sigma_0}}}}.
			\end{align*}
			
			\item  Property \eqref{assumption 2 for h} holds true.
		\end{itemize}
		
	\end{Assumption}

\textbf{Example.} 
Let us here give an example of noise structure satisfying Assumption \ref{Assumption-2}. As above, we only consider the case where $h(t,u)\ {\rm d}\W
	=b(t,u) \ {\rm d}W$ with $W$ a standard 1-D Brownian motion. Let $k\ge1$ and $q(\cdot)$ be a continuous and bounded function, then for any $k,n\ge1$,
	$$b(t,u)= q(t) {\rm e}^{-\frac{1}{\|u\|_{H^{\sigma_0}}}}(1-\partial_{xx}^2)^{-1}\partial_{x}[(u_x)^k+(\HH u_x)^n]$$
	satisfies Assumption \ref{Assumption-2}, where $(1-\partial_{xx}^2)^{-1}$ is defined in \eqref{Helmboltz operator 2}.

	\subsection{Main results and remarks}
   In this subsection, we present the precise statements of the different results shown in this article. The first result reads
	
	\begin{Theorem}\label{Local pathwise solution}
		Let $s>3$ and let $h(t,u)$ satisfy Assumption \ref{Assumption-1}. If $u_0$ is an $H^s$-valued $\mathcal{F}_0$-measurable random variable satisfying $\E\|u_0\|^2_{H^s}<\infty$, then there is a local unique pathwise solution $(u,\tau)$ to \eqref{SCCF problem} in the sense of Definition \ref{pathwise solution definition} with
		\begin{equation}\label{L2 moment bound}
			u(\cdot\wedge \tau)\in L^2\left(\Omega; C\left([0,\infty);H^s\right)\right).
		\end{equation}
		Moreover, $(u,\tau)$ can be extended to a unique maximal pathwise solution $(u,\tau^*)$ with the following blow-up criterion:
		\begin{equation}\label{Blow-up criterion common}
			\textbf{1}_{\left\{\limsup_{t\rightarrow \tau^*}\|u(t)\|_{H^{s}}=\infty\right\}}=\textbf{1}_{\left\{\limsup_{t\rightarrow \tau^*}\|u_x(t)\|_{L^\infty}+\|\HH u_x(t)\|_{L^\infty}=\infty\right\}}\ \p-a.s.
		\end{equation}
	\end{Theorem}

\begin{Remark}\label{Blow-up time remark}
	Before we explain the ideas and difficulties regarding the existence of local pathwise solutions,let us first stress some important remarks on the blow-up criterion \eqref{Blow-up criterion common}. 
	
	\begin{itemize}
		\item The blow-up criterion \eqref{Blow-up criterion common} implies that the 
		$H^{s'}$ norm of the solution within the range $s'\in(\frac32, s]$ blows up at the same time $\tau^*$. Indeed,	for any fixed $s$ and $s'\in(\frac32, s)$, since 
		$$\|u_x\|_{L^{\infty}}+\|\HH u_x\|_{L^{\infty}}\lesssim \|u\|_{H^{s'}}\leq \|u\|_{H^{s}},$$ we can conclude that $\|u\|_{H^{s'}}$ blows up no later than the time $\|u_x(t)\|_{L^{\infty}}+\|\HH u_x\|_{L^{\infty}}$ blows up  but no earlier than the time $\|u\|_{H^{s}}$ blows up. Therefore, equality \eqref{Blow-up criterion common} shows that all of the 	$H^{s'}$ norms have the same blow-up time. This fact will be used to prove Theorem \ref{Non blowup}  (see \eqref{hat tau=tau star}).
		
		\item Invoking the well-known logarithmic Sobolev inequality involving the Hilbert transform $\HH$
		\begin{equation}\label{Dong estimate}
			\|\mathcal{H}u_x\|_{L^\infty}\lesssim\, 
			\left(1+\|u_x\|_{L^\infty}\log\left({\rm e}+\|u_x\|_{H^1}\right)+\|u_x\|_{L^2}\right),
		\end{equation}
		one can show in deterministic case, by using \eqref{Dong estimate}, 
		that the blow-up criterion \eqref{Blow-up criterion common} can be improved into (cf.~\cite{Dong-2008-JFA})
		$$ \limsup_{t\rightarrow \tau^*}\|u(t)\|_{H^s}=\infty
		\Longleftrightarrow  \limsup_{t\rightarrow \tau^*}\|u_x(t)\|_{L^{\infty}}=\infty.$$
		However, it is still not clear how to achieve this in the stochastic setting. Technically, because $\E\left[\left(1+\|u_x\|_{L^\infty}\log\left({\rm e}+\|u_x\|_{H^1}\right)+\|u_x\|_{L^2}\right)\|u\|_{H^s}\right]$ is involved, we have not been able to close the estimate for 
		$\E\|u\|^2_{H^s}$.
	\end{itemize}

\end{Remark}
	
	\begin{Remark}\label{remark:Th1}
		Now we give a remark regarding  the existence part in Theorem \ref{Local pathwise solution}. Since we focus on a Cauchy problem defined on the whole space $\R$, the classical 
		probabilistic compactness argument for non-linear SPDEs in bounded domain seems inapplicable in this work. Indeed, we have the following essential difficulties:

		\begin{itemize}
			\item For smooth $u$, in the \textit{a priori} $L^2(\Omega;H^s)$ estimate for $u$, after using the It\^{o} formula for $\|u\|^2_{H^s}$, we will have to deal with
			$\E \left(f^2(\|u_x\|_{L^\infty})(1+\|u\|^2_{H^s})\right)$ and 
			$\E \left(\left(\|u_x\|_{L^{\infty}}+\|\HH u_x\|_{L^{\infty}}\right)\|u\|^2_{H^s}\right)$, coming from $h(t,u)$ and $(\HH u)u_x$, respectively. To close the estimate, these two terms should be controlled in terms of 
			$\E \left(1+\|u\|^2_{H^s}\right)$. Since we cannot split the mathematical expectation, 
			we add a cut-off function to the original problem to cut the non-linear parts in terms of $\|\cdot\|_{H^r}$ (see \eqref{cut-off problem} below) with suitable $r$ such that $H^{s}\hookrightarrow H^{r}\hookrightarrow \Wlip$. This enables us to close the \textit{a priori} estimate. In this work, $s>3$ and we let $r=s-3/2$.

			\item As a second step, we construct the approximation scheme and obtain certain uniform estimates. Our next aim, is to pass to the limit to obtain the existence of a solution.
			If the target problem is defined on bounded domain, usually one can find $q\in(r, s)$ such that $H^s\hookrightarrow\hookrightarrow H^q\hookrightarrow H^r$ (here $\hookrightarrow\hookrightarrow$ means the embedding is compact). By the uniform estimates, one can follow the compactness argument (cf. Prokhorov's Theorem and Skorokhod's Theorem) to obtain the convergence in $H^q$,  which enables us to pass to the limit to find a martingale solution to the \textit{cut-off} version of the target SPDE. We can \textit{a posteriori} introduce a stopping time to remove the cut-off.  We refer to \cite{Bensoussan-1995-AAM,Brzezniak-Ondrejat-2007-JFA,Debussche-Glatt-Temam-2011-PhyD,Hofmanova-2013-SPTA,GlattHoltz-Vicol-2014-AP,Tang-2018-SIMA,Tang-2020-arxiv} for different examples. On the contrary, in unbounded domains, the compact embedding only holds true for the
			``local" space $H^s_{\rm loc}\hookrightarrow\hookrightarrow H^q_{\rm loc}$. Because the \textit{cut-off} $\|\cdot\|_{H^r}$ appears in the problem itself, even though we obtain the convergence in $H^q_{\rm loc}$ (cf. Prokhorov's Theorem and Skorokhod's Theorem), it is still \textit{not} enough to guarantee the convergence of $\|\cdot\|_{H^r}$ because $\|\cdot\|_{H^r}$  is a \textit{global} object (for space variable), which cannot be controlled by a \textit{local} condition. As mentioned  above, because the \textit{cut-off} is needed,  we have to establish a convergence result in some topology no weaker than $H^{r}$.
			%
			
			\item In this paper, motivated by \cite{Li-Liu-Tang-2021-SPA}, by a careful analysis on the differences between any two approximate solutions, we  will show that there exists a sub-sequence of the approximate solutions converging in $C([0,T];H^{r})$ almost surely. 
			We also remark that our analysis is also available for the torus case, i.e.  $x\in\T=\R/2\pi\mathbb Z$, since all the estimates can be perform in the same way (up to some obvious modifications) for $x\in\T$. In the  deterministic case where the \textit{cut-off} is no longer needed, convergence in local Sobolev spaces is sufficient to take limit to find a solution.

		\end{itemize}

	\end{Remark}

	Hereafter, we consider the noise effect versus finite time blow-up. Our second result gives a partial answer to \ref{Q:noise global solution}:

	\begin{Theorem}[Noise preventing blow-up]\label{Non blowup}  Let $s>3$ and $u_0\in H^s$ be an $H^s$-valued $\mathcal{F}_0$-measurable random variable with $\E\|u_0\|^2_{H^s}<\infty$.  If 
		Assumption \ref{Assumption-alpha} holds true, then
		the corresponding maximal pathwise solution
		$(u,\tau^*)$  to \eqref{SCCF non blow up Eq} satisfies
		$$\p\left\{	\tau^*=\infty\right\}
		=1.$$
	\end{Theorem}

	\begin{Remark}\label{remark on global existence}
		Let us first recall that in the deterministic counterpart of \eqref{SCCF non blow up Eq}, the blow-up of regular enough solutions actually cannot be prevented,  cf.  \cite{Cordoba-etal-2005-Annals,Silvestre-Vicol-TAMS,Li-Rodrigo-2008}. Therefore, Theorem \ref{Non blowup} justifies the idea that fast growing non-linear noise (\textit{strong} noise) can regularize the solutions in terms of preventing singularities. Here we use the terminology ``fast growing" condition, described by \eqref{noise alpha Lyapunov} in Assumption \ref{Assumption-alpha}, to cancel (notice that $\G''<0$ in \eqref{noise alpha Lyapunov}) the growth of the non-local transport term $(\HH u)u$ such that $\E\G(\|u\|^2_{H^s})$ can be controlled with
		a Lyapunov type function $\G$. The idea of using a Lyapunov type function is motivated by the works \cite{Brzezniak-etal-2005-PTRF,Hasminskii-1969-Book, Ren-Tang-Wang-2020-Arxiv,Rohde-Tang-2021-NoDEA}.
	\end{Remark}	
	
	Complementary, in the case of linear noise, we can show that singularities occur in finite time with positive probability which yields a partial answers to \ref{Q:noise blow up}. The precise statement reads
	\begin{Theorem}[Blow-up linear noise]\label{linear:theorem:blowup}
		Let $0<K<1$, $s>3$ and $b(t)$  satisfy Assumption \ref{Assumption-3}. Assume that  $u_0=u_{0}(\omega,x)$ is an $H^s$ valued $\mathcal{F}_{0}$ measurable random variable satisfying $\mathbb{E}\norm{u_{0}}^2_{H^s}< \infty$ and attaining a global maximum at $x_0$. If
		\begin{equation}\label{condition:blow:up}
			\Lambda u_{0}(\omega,x_0)> \frac{b^*}{K}\ \ \p-a.s.,
		\end{equation} 
		where $b^*$ is given in Assumption \ref{Assumption-3}, then the unique maximal pathwise solution $(u,\tau^*)$ to \eqref{linear:SCCF:problem} satisfies that		
		\begin{equation}\label{blow:up}
			\mathbb{P}\{\tau^*<\infty\}\geq\mathbb{P}\{\limsup_{t\rightarrow \tau^*}\max_{x\in \mathbb{R}}\Lambda u(t)=+\infty\}\geq 	\mathbb{P}\{e^{\int_{0}^{t} b(t') \, {\rm d}\W_{t'}}> K \ \forall t \}>0.
		\end{equation}
	\end{Theorem}
	
	\begin{Remark}\label{remark blow-up}	We make the following remarks regarding Theorem \ref{linear:theorem:blowup}:
		\begin{itemize}
			
			\item Motivated by previous works \cite{GlattHoltz-Vicol-2014-AP,Rockner-Zhu-Zhu-2014-SPTA,Rohde-Tang-2021-NoDEA}, to study \eqref{linear:SCCF:problem}, we make use of Girsanov type transform to obtain a PDE with random coefficient instead of a SPDE to study blow-up. In this linear noise case, the blow-up criterion \eqref{Blow-up criterion common} also holds true. 			
			By direct computation, we have $\HH u_x=\Lambda u$ for the fractional Laplace operator $\Lambda=(-\Delta)^{\frac{1}{2}}$. Hence the blow-up criterion \eqref{Blow-up criterion common} becomes 
			\begin{equation*}
				\textbf{1}_{\left\{\limsup_{t\rightarrow \tau^*}\|u(t)\|_{H^{s}}=\infty\right\}}=\textbf{1}_{\left\{\limsup_{t\rightarrow \tau^*}\|u_x(t)\|_{L^\infty}+\|\Lambda u(t)\|_{L^\infty}=\infty\right\}}\ \p-a.s.
			\end{equation*}
	This motivates us to consider blow-up for the particular case $\limsup_{t\rightarrow \tau^*}\|\Lambda u(t)\|_{L^\infty}=\infty$. Actually the case we obtain, as shown in \eqref{blow:up},  is $\limsup_{t\rightarrow \tau^*}\max_{x\in \mathbb{R}}\Lambda u(t)=+\infty$. So far it is still not clear if $\limsup_{t\rightarrow \tau^*}\min_{x\in \mathbb{R}}\Lambda u(t)=-\infty$ holds for blow-up.

			\item  The key identity used in our analysis on blow-up is originally introduced in  \cite[Proposition 3.5]{Silvestre-Vicol-TAMS}, where the authors show that $C^1$ global solutions to \eqref{linear:SCCF:problem} cannot exists.  The precise statement of this identity in our stochastic context is given in \eqref{identity:SV}.  Although the idea towards the proof of \eqref{identity:SV} is similar to \cite{Silvestre-Vicol-TAMS}, we cannot assume without loss of generality (as in \cite[Proposition 3.5]{Silvestre-Vicol-TAMS}) that $\tilde{v}(z_0) = 0$ which simplifies the proof of \eqref{identity:SV}. The main reason relies on the fact that in the stochastic setting $z_0$ depends on $t$ and $\omega\in\Omega$. As a consequence, for different $\omega$, the time $t$ such that $\tilde{v}(z_0) = 0$ may be different, and therefore we cannot assume that $\tilde{v}(z_0) = 0$. On the other hand, if we fix $t$ such that $\tilde{v}(z_0) = 0$, then $\tilde{v}(z_0) = 0$ may not hold almost surely, which brings further obstacles in proving \eqref{identity:SV}.

				\item For trivial $b^*=0$, which means $b(t)\equiv0$, Theorem \ref{linear:theorem:blowup} recovers the deterministic result \cite[Theorem 3.7]{Silvestre-Vicol-TAMS}. Due to the presence of the noise term $b(t)u {\rm d}W$, the condition \eqref{condition:blow:up} is more restrictive compared to the deterministic case where $\Lambda u_{0}(x_0)>0$ suffices to show the finite time blow-up, \cite[Theorem 3.7]{Silvestre-Vicol-TAMS}. Therefore, although finite time singularities can be shown in the stochastic case it is somehow restrictive and strongly depends on the choice of the coefficient $b(t)$. 
			
		\end{itemize}
		
	\end{Remark}

	Regarding question \ref{Q:noise dependence}, we have the following partial (negative) answer:
	\begin{Theorem}[Weak instability]\label{Weak instability}
		Let $\s=(\Omega, \mathcal{F},\{\mathcal{F}_t\}_{t\geq0},\p, \W)$ be a fixed stochastic basis and $s>3$. If $h$ satisfies Assumption \ref{Assumption-2}, then at least one of the following properties holds true.
		\begin{enumerate}
			\item For any $R\gg 1$, the $R$-exiting time is not strongly stable for the zero solution to \eqref{SCCF problem} in the sense of Definition \ref{Definition stability of exiting time};
			\item There is a $T>0$ such that solution map $u_0\mapsto u$ defined by \eqref{SCCF problem} is not uniformly continuous as a map from $L^\infty(\Omega,H^s)$ into $L^1(\Omega;C([0,T],H^s))$. 
			More precisely, there exist two sequences of solutions $u^{1,n}(t,x)$ and $u^{2,n}(t,x)$, and two sequences of stopping time $\tau_{1,n}$ and $\tau_{2,n}$, such that
			\begin{itemize}
				\item For $i=1,2$, $\p\{\tau_{i,n}>0\}=1$ for each $n>1$. Besides,
				\begin{equation}\label{tau 1 2 n}
					\lim_{n\rightarrow\infty}\tau_{1,n}=\lim_{n\rightarrow\infty}\tau_{2,n}=\infty\ \ \p-a.s.
				\end{equation}
				\item For $i=1,2$, $u^{i,n}\in C([0,\tau_{i,n}],H^s)$ $\p-a.s.$, and
				\begin{equation}\label{sup u}
					\E\left(\sup_{t\in[0,\tau_{1,n}]}\|u^{1,n}(t)\|_{H^s}+\sup_{t\in[0,\tau_{2,n}]}\|u^{2,n}(t)\|_{H^s}\right)\lesssim 1.
				\end{equation}
				\item At initial time $t=0$,
				\begin{equation}\label{same initail data}
					\lim_{n\rightarrow\infty}\E\|u^{1,n}(0)-u^{2,n}(0)\|_{H^s}^2=0.
				\end{equation}
				\item When $t>0$,
				\begin{equation}\label{sup sin t}
					\liminf_{n\rightarrow\infty}\E\sup_{t\in[0,T\wedge\tau_{1,n}\wedge\tau_{2,n}]}
					\|u^{1,n}(t)-u^{2,n}(t)\|_{H^s}
					\gtrsim \sup_{t\in[0,T]}|\sin(t)|.
				\end{equation}
			\end{itemize}
		\end{enumerate}
		
	\end{Theorem}

	\begin{Remark}\label{remark on weak instability}
	The idea behind Theorem \ref{Weak instability} can be understood as follows: one cannot improve the continuity of the map $u_{0}\mapsto u$, and simultaneously, the		
		stability of the exiting time at $u\equiv0$. 
		Hereafter, we briefly outline the main difficulties encountered in the proof of Theorem \ref{Weak instability} and the main strategies used to tackle them.
		\begin{itemize}
			
			\item For system \eqref{SCCF problem}, we do not know how to obtain any explicit expression of the solutions. Therefore, to establish \eqref{sup sin t}, the idea relies on finding two sequences of approximate solutions $\{u_{m,n}\}$ ($m\in\{1,2\}$)  such that the difference between $u_{m,n}$ and the actual solution $u^{m,n}$ tends to zero as $n\to\infty$, from which one can prove \eqref{sup sin t} by using $\{u_{m,n}\}$ rather than $\{u^{m,n}\}$. In this way, the first difficulty lies in the construction of such approximate solutions $\{u_{m,n}\}$.
			In this work, by some delicate calculation, we are able to construct two sequences of approximate solutions $\{u_{m,n}\}$ such that the actual solutions $\{u^{m,n}\}$ satisfy 
			\begin{align}
				u^{m,n}(0)=u_{m,n}(0),\ \ 	\lim_{n\rightarrow\infty}\E\sup_{[0,\tau_{m,n}]}\|u^{m,n}-u_{m,n}\|_{H^s}
				=0,\label{non uniform remark equ}
			\end{align}
			where $u^{m,n}$ exists at least on $[0,\tau_{m,n}]$. Technically, since the equation involves the non-local Hilbert transform $\HH$ and the problem is defined on $\R$, the construction of approximate solutions $u_{m,n}$ is much more elaborated than the constructions in \cite{Miao-Rohde-Tang-2021-arXiv,Rohde-Tang-2020-JDDE,Tang-2020-arxiv}.

			\item The second difficulty we have to surpass is that we need to guarantee that $\inf_{n}\tau_{m,n}>0$ almost surely when dealing with \eqref{non uniform remark equ}. This obstacle  comes from the lack of lifespan estimates which we believe is quite a common issue in SPDEs. Indeed, in
			deterministic cases, one can easily obtain the lifespan estimate, from which it is not difficult to find a common interval $[0,T]$ such that all actual solutions exist on $[0,T]$ (see for example Lemma  \ref{lemma ul Hr}). In the stochastic setting, we are not able to show the precise estimate. The key observation to surpass this difficulty is that, the property $\inf_{n}\tau_{m,n}>0$ can be connected to the stability property of the exiting time (see Definition \ref{Definition stability of exiting time}). The condition comprising the fact that the $R_0$-exiting time is strongly stable at the zero solution will be used to provide a common existence time $T>0$ such that for all $n$, $u^{m,n} $ exists up to $T$ (see Lemma \ref{exiting time infty lemma} below). Therefore, to prove Theorem \ref{Weak instability}, we will show that, if the $R_0$-exiting time is strongly stable at the zero solution for some $R_0\gg1$,  then the solution map $u_0\mapsto u$ defined by \eqref{SCCF problem} cannot be uniformly continuous.

		\end{itemize}

	\end{Remark}

	\subsection*{Plan of the paper}
	
	We outline the structure of the paper.  In Section \ref{Preliminaries} we provide some relevant preliminaries and recall well-known estimates that will be employed throughout the paper. In Section \ref{local existence section} we show the existence and uniqueness of local pathwise solutions and derived a blow-up criterion proving Theorem \ref{Local pathwise solution}. We will divide the proof into several subsections. First we use an approximation scheme and perform uniform bounds in Subsection \ref{section:appro and estimate}. Next, in Subsection \ref{subsec:convergence} we show the convergence of the approximated solutions and afterwards in Subsection \ref{subsec:global:path} we prove the global pathwise solutions to the cur-off problem. We conclude the proof of Theorem \ref{Local pathwise solution} in Subsection \ref{Local pathwise solution}. Section \ref{Section:global:strong} is devoted to study the effect of strong noise and prove Theorem \ref{Non blowup}. In Section \ref{Sec:blow:up:linear} we show that finite time singularity occur in the case of linear noise thus proving Theorem \ref{linear:theorem:blowup}. To conclude the article, in Section \ref{sect:weak instability}, we prove Theorem \ref{Weak instability}. The proof is divided into Subsection \ref{subsec:approximate:weak} - Subsection \ref{conclude:weak:theorem}.

	\section{Preliminaries}\label{Preliminaries}
	For any $\varepsilon\in(0, 1)$, $J_{\varepsilon}$ is the Friedrichs mollifier defined by $J_{\varepsilon}f(x)=j_{\varepsilon}\star f(x)$,
	where $\star$ stands for the convolution, $j_{\varepsilon}(x)=\frac{1}{\e}j(\frac{x}{\e})$ and $j(x)$ is a Schwartz function satisfying $0\leq\widehat{j}(\xi)\leq1$ for all $\xi\in \R$ and $\widehat{j}(\xi)=1$ for any $|\xi|\leq1$, where $\widehat{f}$ denotes the Fourier transform of $f$.
	It is obvious that  $\widehat{j}_{\varepsilon}(\xi)=\widehat{j}(\varepsilon \xi)$.  Moreover, for any $u\in H^s$, we have, cf. \cite{Tang-2018-SIMA,Tang-2020-arxiv},
	\begin{align}
		\|u-J_{\varepsilon}u\|_{H^r}&\sim o(\varepsilon^{s-r}),\ \ r\leq s,\label{mollifier property 1}\\
		\|J_{\varepsilon}u\|_{H^{r}}&\lesssim  O(\varepsilon^{s-r})\|u\|_{H^{s}},\ \ r> s,\label{mollifier property 2}\\
		(J_{\varepsilon}f, g)_{L^2}&=(f, J_{\varepsilon}g)_{L^2}.\label{mollifier property 4}
	\end{align}
	The Hilbert transform is defined as 
	\begin{equation}\label{Hilbert:transform}
		\HH f(x)= \frac{1}{\pi}\mbox{p.v.}\int_{\mathbb{R}}\frac{f(y)}{y-x} \, {\rm d}y.
	\end{equation}
	The fractional differential operator $\Lambda^{\alpha}=(-\Delta)^{\frac{\alpha}{2}}$ is defined as the following singular integral operator 
	\begin{equation}\label{frac:integral:rep}
		\Lambda^{\alpha} f(x)= c_{\alpha} \mbox{p.v.} \int_{\mathbb{R}}\frac{f(x)-f(y)}{|x-y|^{1+\alpha}}\, {\rm d}y,
	\end{equation}
	where the  constant $c_\alpha=\frac{4^{\frac{\alpha}{2}}\Gamma(\frac{1}{2}+\alpha/2)}{\sqrt{\pi}\abs{\Gamma(-\alpha/2)}}$ is a normalization constant and $\Gamma$ represents the classical gamma function. In particular when $\alpha=1$, $\HH f_{x}=-\Lambda f.$ Since $J_\e$, $\HH$ and $D^s$ can be characterized by their Fourier multipliers, it is  easy to  see 
	\begin{equation}\label{H Je Ds}
		[D^s,\mathcal{H}]=[D^s,J_\e]=[\partial_x,\mathcal{H}]=[\partial_x,J_\e]=[J_\e,\mathcal{H}]=0,
	\end{equation}
	and for any $s\geq 0$,
	\begin{align}
		\|\HH u\|_{H^s}\leq \|u\|_{H^s},\ 
		\|J_{\varepsilon}u\|_{H^s}\leq \|u\|_{H^s}.\label{H Je Hs norm}
	\end{align}

	\begin{Lemma}[Page 3 in \cite{Taylor-2011-PDEbook3}]\label{Je commutator} Let $J_\e$ be defined as in the above. Assume $g\in W^{1,\infty}$ and $f\in L^2$. Then for some $C>0$,
		\begin{align*}
			\|[J_{\varepsilon}, g]\partial_xf\|_{L^2}
			\leq C\|\partial_x g\|_{L^\infty}\|f\|_{L^2}.
		\end{align*}
	\end{Lemma}

	We also recall the following well-known estimates.
	\begin{Lemma}[\cite{Kato-Ponce-1988-CPAM,Kenig-Ponce-Vega-1991-JAMS}]\label{Kato-Ponce commutator estimate}
		If $f,g\in H^s\bigcap W^{1,\infty}$ with $s>0$, then for $p,p_i\in(1,\infty)$ with $i=2,3$ and 
		$\frac{1}{p}=\frac{1}{p_1}+\frac{1}{p_2}=\frac{1}{p_3}+\frac{1}{p_4}$, we have
		$$
		\|\left[D^s,f\right]g\|_{L^p}\leq C_{s,p}(\|\nabla f\|_{L^{p_1}}\|D^{s-1}g\|_{L^{p_2}}+\|D^sf\|_{L^{p_3}}\|g\|_{L^{p_4}}),$$
		and
		$$\|D^s(fg)\|_{L^p}\leq C_{s,p}(\|f\|_{L^{p_1}}\|D^s g\|_{L^{p_2}}+\|D^s f\|_{L^{p_3}}\|g\|_{L^{p_4}}).$$
	\end{Lemma}
\begin{Lemma}\label{Huux Hs inner product}
	Let $s>3$. Let  $J_\e$ be the Friedrichs mollifier defined before. 
	There is a constant $Q=Q(s)>0$ such that
	\begin{equation}
		\left|\left((\HH u)u_x, u\right)_{H^{s-1}}\right|
		\leq Q(\|u_x\|_{L^\infty}+\|\HH u_x\|_{L^\infty})\|u\|^2_{H^{s-1}},\label{Huux u Hs-1}
		\end{equation}
	and for all $\e>0$,
	\begin{equation}
		\left|\left(J_\e  \left[(\HH u)u_x\right], J_\e  u\right)_{H^s}\right|
		\leq Q(\|u_x\|_{L^\infty}+\|\HH u_x\|_{L^\infty})\|u\|^2_{H^s}\label{Huux u Hs}.
	\end{equation}
\end{Lemma}
\begin{proof}
	We first prove \eqref{Huux u Hs}.
	Due to  \eqref{H Je Ds} and \eqref{mollifier property 4}, we commute the operator  to derive
	\begin{align*}
		&\left(D^sJ_\e  
		\left[(\HH u)u_x\right],D^sJ_\e  u\right)_{L^2}\notag\\
		=&\left(\left[D^s,\HH u\right]u_x,D^sJ^2_\e u\right)_{L^2}+
		\left([J_\e ,\HH u]D^su_x, D^sJ_\e  u\right)_{L^2}
		+\left((\HH u)D^sJ_\e  u_x, D^sJ_\e  u\right)_{L^2}.
	\end{align*}
	Then it follows from Lemmas   \ref{Je commutator} and \ref{Kato-Ponce commutator estimate}, integration by parts, \eqref{H Je Hs norm} and $H^s\hookrightarrow W^{1,\infty}$ that
	\begin{equation*}
		\left|\left(\left[D^s,\HH u\right]u_x,D^sJ^2_\e u\right)_{L^2}\right|
		\lesssim
		\left(\|u_x\|_{L^{\infty}}+\|\HH u_x\|_{L^{\infty}}\right)\|u\|^2_{H^s},
	\end{equation*}
	\begin{equation*}
		\left|\left([J_\e ,\HH u]D^su_x, D^sJ_\e  u\right)_{L^2}\right|
		\lesssim\ \|\HH u_x\|_{L^{\infty}} \|u\|^2_{H^s},
	\end{equation*}
	and 
	\begin{equation*}
		\left|\left((\HH u)D^sJ_\e  u_x, D^sJ_\e  u\right)_{L^2}\right|
		\lesssim\ \|\HH u_x\|_{L^{\infty}} \|u\|^2_{H^s}.
	\end{equation*}
	Combining the above inequalities gives rise to \eqref{Huux u Hs}. For \eqref{Huux u Hs-1}, since $\left((\HH u)u_x, u\right)_{H^{s-1}}$ is well-defined in this case, by repeating the above analysis, one can easily obtain the \eqref{Huux u Hs-1} and the proof is therefore completed.
\end{proof}

\begin{Remark}
We remark that \eqref{Huux u Hs} will be used in the proof of blow-up criterion \eqref{Blow-up criterion common} (see \eqref{Huux u blow-up} below) and \eqref{Huux u Hs-1} will be used in the proof of Theorem \ref{Non blowup} (see \eqref{Huux u global existence} below). But it is worthwhile pointing out that \eqref{Huux u Hs-1} is also used implicitly in the proof of Theorem \ref{Non blowup}. We refer to Remark \ref{Remark global existence} for more details.
\end{Remark}

	Let us collect some identities and formulas regarding the Hilbert transform and the fractional Laplacian operator, cf. \cite{Silvestre-Vicol-TAMS,Cordoba-etal-2005-Annals}.
	The first one is the so-called Cotlar's identity
	\begin{equation}\label{hilbert:iden:2}
		2\HH \left( f (\HH f) \right)  = (\HH f)^{2}-f^{2}.
	\end{equation}

The second one is the following equality:

	\begin{Lemma}[Corollary 3.3, \cite{Silvestre-Vicol-TAMS}]\label{hilbert:iden:3}
		For any $f\in L^1(\mathbb{R};\mathbb{R})$, we have that for $g(x)=xf(x)$,
		\begin{equation*}
			\Lambda g (x) = x \Lambda f(x) -  \HH f(x) 
		\end{equation*}
	\end{Lemma}

	Finally, we recall the following estimate on the product of a  Schwartz function and a trigonometric function.
	\begin{Lemma}[\cite{Himonas-Kenig-2009-DIE}]\label{lemma uh Hr}
		Let $\delta>0$ and $\alpha\in\R$. Then for any $r\geq 0$ and any Schwartz function $\psi$, the following equation holds true:
		\begin{equation}\label{uh Hr estimate}
			\lim_{n\rightarrow\infty}n^{-\frac{\delta}{2}-r}\left\|\psi\left(\frac{x}{n^{\delta}}\right)\cos(nx-\alpha)\right\|_{H^r}=\frac{1}{\sqrt{2}}\|\psi\|_{L^2}.
		\end{equation}
		Moreover,  \eqref{uh Hr estimate}  also holds true when $\cos$ is replaced by $\sin$.
	\end{Lemma}

	\section{Proof of Theorem \ref{Local pathwise solution}}\label{local existence section}

	For the sake of clarity, the proof is divided into several subsections.

	\subsection{Approximation scheme and uniform estimates}\label{section:appro and estimate}

	%
	%
	
	The first step is to construct a suitable approximation scheme. For any $R>1$, we let $\chi_R(x):[0,\infty)\rightarrow[0,1]$ be a $C^{\infty}$ function such that $\chi_R(x)=1$ for $x\in[0,R]$ and $\chi_R(x)=0$ for $x>2R$.
	Then we consider the following cut-off problem on $\R$,
	\begin{equation} \label{cut-off problem}
		\left\{\begin{aligned}
			&{\rm d}u+\chi_R(\|u\|_{H^{s-3/2}})(\HH u)u_x{\rm d}t=\chi_R(\|u\|_{H^{s-3/2}})h(t,u){\rm d}\W,\\
			&u(\omega,0,x)=u_0(\omega,x)\in H^{s},\ \ s>3.
		\end{aligned} \right.
	\end{equation}
	To apply the theory of SDEs in Hilbert space to \eqref{cut-off problem}, we will have to mollify the transport term $(\HH u)u_x$ since the product $(\HH u)u_x$ loses regularity. To this end, we consider the following approximation scheme:
	\begin{equation} \label{approximate problem}
		\left\{\begin{aligned}
			{\rm d}u+G_{1,\e}(u){\rm d}t&=G_{2}(t,u){\rm d}\W,\\
			G_{1,\e}(u)&=\chi_R(\|u\|_{H^{s-3/2}})J_\e [\left(\HH J_\e u \right)\partial_{x}J_\e u],\\
			G_{2}(t,u)&=\chi_R(\|u\|_{H^{s-3/2}})h(t,u),\\
			u(0,x)&=u_0(x)\in H^{s},
		\end{aligned} \right.
	\end{equation}
	where $J_{\e}$ is the Friedrichs mollifier defined in Section \ref{Preliminaries}. 
	
	After mollifying the non-local transport term $(\HH u)u_x$,  we see that $G_{1,\e}(\cdot)$ and $G_2(t,\cdot)$ are locally Lipschitz continuous in $H^s$. Moreover, the cut-off function $\chi_R(\|\cdot\|_{H^{s-3/2}})$ gives the linear growth condition (cf, Lemma \ref{Huux Hs inner product} and \eqref{assumption 1 for h}), i.e., there are constants $l_1=l_1(\e,R)>0$ and $l_2=l_2(R)>0$ such that for all $t\ge0$ and $s>3$,
	\begin{equation} \label{H1 H2 condition}
		\|G_{1,\e}(u)\|_{H^s}\leq l_1(1+\|u\|_{H^s}),\ \
		\|G_{2}(t,u)\|_{\LL_2(\U;H^s)}\leq l_2(1+\|u\|_{H^s}),\ \ t\in[0,T].
	\end{equation}

	Therefore,  for a fixed stochastic basis $\s=(\Omega, \mathcal{F},\p,\{\mathcal{F}_t\}_{t\geq0}, \W)$ and for $u_0\in L^2(\Omega;H^s)$ with $s>3$, the existence theory of SDEs in Hilbert space  (see for example \cite[Theorem 4.2.4 with Example 4.1.3]{Prevot-Rockner-2007-book} and \cite{Kallianpur-Xiong-1995-book}), \eqref{approximate problem} admits a unique solution $u_{\e}\in C([0,\infty ),H^s)$ $\p-a.s.$   
	
	We have the following uniform-in-$\e$ estimate:
	
	\begin{Proposition}\label{global solution to appro and estimates}
		Let $\s=(\Omega, \mathcal{F},\p,\{\mathcal{F}_t\}_{t\geq0}, \W)$ be a fixed stochastic basis. Let $s>3$, $R>1$ and $\e\in(0,1)$. Assume $h$ satisfies Assumption \ref{Assumption-1} and $u_0\in L^2(\Omega;H^s)$ is an $H^s$-valued $\mathcal{F}_0$-measurable random variable. Let $u_{\e}\in C([0,\infty);H^{s})$ solve \eqref{approximate problem} $\p-a.s.$, then  for any $T>0$, 
		there are $C=C(R,T,u_0)>0$ $(i=0,1,2,3,4)$ such that
		\begin{align}
			\sup_{\e>0}\E\sup_{t\in[0,T]}\|u_{\varepsilon}(t)\|^2_{H^s}
			\leq C.\label{E r e}
		\end{align}
		
		\begin{proof}
			Using the It\^{o} formula for $\|u_\e\|^2_{H^s}$, we have that for any $t>0$,	
			\begin{align}
				{\rm d}\|u_\e(t)\|^2_{H^s}
				= \, & 2\chi_R(\|u_\e\|_{H^{s-3/2}})\left( h(t,u_\e){\rm d}\W,u_\e \right)_{H^s}\notag\\
				&-2\chi_R(\|u_\e\|_{H^{s-3/2}})
				\left(D^sJ_{\varepsilon}\left[
				J_{\varepsilon}(\HH u_\e)\partial_xJ_{\varepsilon}u_\e\right],D^s u_\e \right)_{L^2}{\rm d}t\notag\\
				&+ \chi^2_R(\|u_\e\|_{H^{s-3/2}})\|h(t,u_\e)\|_{\LL_2(\U; H^s)}^2{\rm d}t.\notag
			\end{align}
		Then, by means of the BDG inequality, \eqref{assumption 1 for h},  $H^{s-3/2}\hookrightarrow W^{1,\infty}$ and \eqref{H Je Hs norm}, we find that for some constant $C=C(R)>0$,
			\begin{align*}
				&\E\sup_{t\in[0,T]}\|u_\e(t)\|^2_{H^s}-\E\|u_0\|^2_{H^s}\\
				\lesssim&
				\E\left(\int_0^{T}\chi^2_R(\|u_\e\|_{H^{s-3/2}})
				f^2(2\|u_\e\|_{H^{s-3/2}})(1+\|u_\e\|^2_{H^s})\|u_\e\|^2_{H^s}{\rm d}t\right)^{\frac12}\\
				&+2\E\int_0^{T}\chi_R(\|u_\e\|_{H^{s-3/2}})\left|\left(D^s J_{\varepsilon}\left[
				J_{\varepsilon}(\HH u_\e)\partial_xJ_{\varepsilon}u_\e\right],D^s u_\e \right)_{L^2}\right|{\rm d}t\\
				&+\E\int_0^{T}\chi^2_R(\|u_\e\|_{H^{s-3/2}})f^2(2\|u_\e\|_{H^{s-3/2}})(1+\|u_\e\|^2_{H^s}){\rm d}t\\
				\leq&\frac12\E\sup_{t\in[0,T]}\|u_\e\|_{H^s}^2+
				C(R)\E \int_0^{T}\left(1+\|u_\e\|^2_{H^s}\right){\rm d}t \\
				&+2\E\int_0^{T}\chi_R(\|u_\e\|_{H^{s-3/2}})\left|\left(D^sJ_{\varepsilon}\left[
				J_{\varepsilon}(\HH u_\e)\partial_xJ_{\varepsilon}u_\e\right],D^s u_\e \right)_{L^2}\right|{\rm d}t.
			\end{align*}
			Let $J_\e u_\e=v$.
			It follows from \eqref{mollifier property 4}, Lemma \ref{Kato-Ponce commutator estimate}, integration by parts and \eqref{H Je Hs norm} that
			\begin{align*}
				\left|\left(D^sJ_{\varepsilon}\left[
				J_{\varepsilon}(\HH u_\e)\partial_xJ_{\varepsilon}u_\e\right],D^s u_\e \right)_{L^2}\right|
				\leq&\left|\left(\left[D^s,
				\HH v \right]v_x,D^sv\right)_{L^2}\right|+
				\left|\left((\HH v ) D^s v_x, D^sv\right)_{L^2}\right|\\
				\leq&  C\left(\|\HH v\|_{H^s}\|v_x\|_{L^\infty}+\|\HH v_x\|_{L^\infty}\|v\|_{H^s}\right)\|v\|_{H^s}\\
				\leq&  C\|v\|_{H^{s-3/2}}\|v\|^2_{H^s},
			\end{align*}
			which implies
			\begin{align*}
				\E\int_0^{T}	\chi_R(\|u_\e\|_{H^{s-3/2}})\left|\left(D^sJ_{\varepsilon}\left[
				J_{\varepsilon}(\HH u_\e)\partial_xJ_{\varepsilon}u_\e\right],D^s u_\e \right)_{L^2}\right|{\rm d}t\leq C(R)\E\int_0^{T} \|u_\e\|_{H^s}^2{\rm d}t.
			\end{align*}
			Therefore we obtain
			\begin{align*}
				\E\sup_{t\in[0,T]}\|u_\e(t)\|^2_{H^s}
				\leq 2\E\|u_0\|^2_{H^s}+ C(R)\int_0^{T} \left(1+\E\sup_{t'\in[0,t]}\|u(t')\|_{H^s}^2\right){\rm d}t.
			\end{align*}
			Using Gr\"{o}nwall's inequality to the above estimate implies that for some $C=C(R,T,u_0)>0$,
			\begin{align*}
				\E\sup_{t\in[0,T]}\|u_\e(t)\|^2_{H^s}\leq C(R,T,u_0),
			\end{align*}
			which is \eqref{E r e}.
		\end{proof}

	\end{Proposition} 
	
	\subsection{Convergence of approximate solutions}\label{subsec:convergence}

	Now we are going to show that, there is a subsequence of $u_\e$ converging in $C([0,T],H^{s-3/2})$ almost surely. To this end, for the solutions $u_\e$ and $u_\eta$ to \eqref{approximate problem}, we consider the following problem for $v_{\e,\eta}=u_\e-u_\eta$,
	\begin{equation}\label{v ee equation}
		{\rm d}v_{\e,\eta}+\left[G_{1,\e}(u_\e)-G_{1,\eta}(u_{\eta})\right]{\rm d}t=\left[G_{2}(t,u_\e)-G_{2}(t,u_{\eta})\right]{\rm d}\W,\ \ v_{\e,\eta}(0)=0.
	\end{equation}
	We notice that
	\begin{align}
		&G_{1,\e}(u_\e)-G_{1,\eta}(u_{\eta})\notag\\
		=&\chi_R(\|u_\e\|_{H^{s-3/2}})
		\left[J_{\e}\left(J_{\e}(\HH u_\e)\partial_xJ_{\e}u_\e\right)
		\right]
		-\chi_R(\|u_\eta\|_{H^{s-3/2}})
		\left[J_{\eta}\left(J_\eta (\HH u_\eta)\partial_xJ_\eta u_\eta\right)
		\right]\notag\\
		=& \left[\chi_R(\|u_{\e}\|_{H^{s-3/2}})-\chi_R(\|u_{\eta}\|_{H^{s-3/2}})\right]
		J_\e[J_\e (\HH u_\e)\partial_xJ_\e u_\e]\notag\\
		&+\chi_R(\|u_{\eta}\|_{H^{s-3/2}})
		(J_\e-J_\eta)[J_\e (\HH u_\e)\partial_xJ_\e u_\e]\notag\\
		&+\chi_R(\|u_{\eta}\|_{H^{s-3/2}})
		J_\eta[(J_\e-J_\eta)(\HH u_\e)\partial_xJ_\e u_\e]+\chi_R(\|u_{\eta}\|_{H^{s-3/2}})
		J_\eta[J_\eta(\HH u_\e-\HH u_\eta)\partial_xJ_\e u_\e]\notag\\
		&+\chi_R(\|u_{\eta}\|_{H^{s-3/2}})
		J_\eta[J_\eta (\HH u_{\eta})\partial_x(J_\e-J_\eta) u_\e]
		+\chi_R(\|u_{\eta}\|_{H^{s-3/2}})
		J_\eta[J_\eta (\HH u_{\eta})\partial_xJ_\eta (u_\e-u_\eta)]\notag\\
		\triangleq&\sum_{i=1}^{6}q_i.\label{G1-G1}
	\end{align}
	and
	\begin{align}
		&G_{2}(t,u_\e)-G_{2}(t,u_{\eta})\notag\\
		=&\chi_R(\|u_\e\|_{H^{s-3/2}})h(t,u_\e)
		-\chi_R(\|u_\eta\|_{H^{s-3/2}})h(t,u_\eta)\notag\\
		=&\left[\chi_R(\|u_{\e}\|_{H^{s-3/2}})-\chi_R(\|u_{\eta}\|_{H^{s-3/2}})\right]
		h(t,u_\e)+\chi_R(\|u_{\eta}\|_{H^{s-3/2}})[h(t,u_\e)-h(t,u_\eta)]\notag\\
		\triangleq&\sum_{i=7}^{8}q_i.\label{G2-G2}
	\end{align}
Invoking It\^{o}'s formula in \eqref{v ee equation} and recalling \eqref{G1-G1}, \eqref{G2-G2} we find that for any $t>0$,
	\begin{align}
		\|v_{\e,\eta}(t)\|^2_{H^{s-3/2}}
		=&\mathcal{Q}_1
		-\int_0^t\mathcal{Q}_2{\rm d}t'
		+\int_0^t\mathcal{Q}_3{\rm d}t',\label{v-ee Hs-3}
	\end{align}
	where 
	\begin{equation}\label{Ni qi}
		\mathcal{Q}_1=2\int_0^t\left( \sum_{i=7}^{8}q_i {\rm d}\W,
		v_{\e,\eta}\right)_{H^{s-3/2}},\ \ 
		\mathcal{Q}_2=2\sum_{i=1}^6\left(q_i,v_{\e,\eta}\right)_{H^{s-3/2}},\ \ 
		\mathcal{Q}_3=\left\|\sum_{i=7}^{8} q_i\right\|_{\LL_2\left(U;H^{s-3/2}\right)}^2.
	\end{equation}
			
			\begin{Lemma}\label{N2 Lemma}
				Let $s>3$. For any $\e,\eta\in(0,1)$, there is a constant $C>0$ such that $\mathcal{Q}_2$ given by \eqref{Ni qi} satisfies
				\begin{align*}
					|\mathcal{Q}_2|
					\leq& C(1+\|u_\e\|^2_{H^s}+\|u_\eta\|^2_{H^s})\|v_{\e,\eta}\|^2_{H^{s-3/2}}+C(\|u_\e\|^4_{H^s}+\|u_\eta\|^4_{H^s})\max\{\e,\eta\}.
				\end{align*}
			\end{Lemma}
			\begin{proof}
				Using the mean value theorem for $\chi_R(\cdot)$ and \eqref{H Je Hs norm}, we have
				\begin{align*}
					\left\| q_1\right\|_{H^{s-3/2}}
					\lesssim\|v_{\e,\eta}\|_{H^{s-3/2}}\|u_\e\|^2_{H^{s}},
				\end{align*}
				Using \eqref{mollifier property 1} and \eqref{H Je Hs norm}, we see that
				\begin{align*}
					\left\| q_i\right\|_{H^{s-3/2}}\lesssim&\max\{\e^{1/2},\eta^{1/2}\}\|u_\e\|^2_{H^{s}},\ i=2,3,\\
					\left\| q_4\right\|_{H^{s-3/2}}		\lesssim&\|v_{\e,\eta}\|_{H^{s-3/2}}\|u_\e\|_{H^{s}},\\
					\left\| q_5\right\|_{H^{s-3/2}}\lesssim&\max\{\e^{1/2},\eta^{1/2}\}\|u_\e\|_{H^{s}}\|u_\eta\|_{H^{s}}.
				\end{align*}
				For $q_6$, using \eqref{mollifier property 4}, \eqref{H Je Ds} and then integrating by part, we have
				\begin{align*}
					\left(q_6,v_{\e,\eta}\right)_{H^{s-3/2}}
					=&\chi_R(\|u_{\eta}\|_{H^{s-3/2}})\int_\R D^{s-3/2}[J_\eta (\HH u_{\eta})\partial_xJ_\eta v_{\e,\eta}]\cdot D^{s-3/2}J_\eta v_{\e,\eta}{\rm d}x\\
					=&\chi_R(\|u_{\eta}\|_{H^{s-3/2}})\int_\R [D^{s-3/2},J_\eta (\HH u_{\eta})]\partial_xJ_\eta v_{\e,\eta}\cdot D^{s-3/2}J_\eta v_{\e,\eta}{\rm d}x\\
					&+\chi_R(\|u_{\eta}\|_{H^{s-3/2}})\int_\R J_\eta (\HH u_{\eta})\partial_xD^{s-3/2}J_\eta v_{\e,\eta}\cdot D^{s-3/2}J_\eta v_{\e,\eta}{\rm d}x\\
					=&\chi_R(\|u_{\eta}\|_{H^{s-3/2}})\int_\R [D^{s-3/2},J_\eta (\HH u_{\eta})]\partial_xJ_\eta v_{\e,\eta}\cdot D^{s-3/2}J_\eta v_{\e,\eta}{\rm d}x\\
					&-\frac12\chi_R(\|u_{\eta}\|_{H^{s-3/2}})\int_\R J_\eta (\HH u_{\eta})_x\cdot (D^{s-3/2}J_\eta v_{\e,\eta})^2{\rm d}x.
				\end{align*}
				Using $\chi_R(\cdot)\leq1$, Lemma \ref{Kato-Ponce commutator estimate}, \eqref{H Je Hs norm} and the embedding $H^{s-3/2}\hookrightarrow W^{1,\infty}$, we have
				\begin{align*}
					\left(q_6,v_{\e,\eta}\right)_{H^{s-3/2}}
					\lesssim&\|u_\eta\|_{H^{s-3/2}}\|\partial_xJ_\eta v_{\e,\eta}\|_{L^{\infty}}\|v_{\e,\eta}\|_{H^{s-3/2}}
					+\|\partial_xJ_\eta (\HH u_\eta)\|_{L^{\infty}}\|v_{\e,\eta}\|^2_{H^{s-3/2}}
					\\
					\lesssim&\|u_\eta\|_{H^{s}}\|v_{\e,\eta}\|^{2}_{H^{s-3/2}}.
				\end{align*}
				Therefore we can put these all together to see that there is a constant $C>0$ such that
				\begin{align*}
					|\mathcal{Q}_2|
					\leq& C(1+\|u_\e\|^2_{H^s}+\|u_\eta\|^2_{H^s})\|v_{\e,\eta}\|^2_{H^{s-3/2}}+C(\|u_\e\|^4_{H^s}+\|u_\eta\|^4_{H^s})\max\{\e,\eta\},
				\end{align*}
				which is the desired estimate.
			\end{proof}

			\begin{Lemma}\label{Convergene of ue probability}
				Let $\s=(\Omega, \mathcal{F},\p,\{\mathcal{F}_t\}_{t\geq0}, \W)$ be a fixed stochastic basis. Let $s>3$, $R>1$ and $\e\in(0,1)$.  Let $u_{\e}\in C([0,\infty);H^{s})$ solve \eqref{approximate problem} $\p-a.s.$	For any $T>0$ and $K>1$, we define
				\begin{equation} \label{tau-e,T}
					\tau^{T}_{\e,K}=\inf\left\{t\geq0:\|u_\e(t)\|_{H^{s}}\geq K\right\}\wedge T,
				\end{equation}
				and 
				\begin{equation} \label{tau-e,eta,T}
					\tau^T_{\e,\eta,K}
					=\tau^{T}_{\e,K}\wedge\tau^{T}_{\eta,K}.
				\end{equation}
				Then we have
				\begin{equation} \label{Cauchy H s-3/2}
					\lim_{\e\rightarrow0}\sup_{\eta\leq\e}
					\E\sup_{t\in[0,\tau^T_{\e,\eta,K}]}\|u_\e-u_\eta\|_{H^{s-3/2}}=0,\ \ K>1.
				\end{equation}
			\end{Lemma}
			\begin{proof}
				Recalling \eqref{v-ee Hs-3} and \eqref{Ni qi}, we have
				\begin{align}
					\|v_{\e,\eta}(t)\|^2_{H^{s-3/2}}
					\leq\ &|\mathcal{Q}_1|+
					\int_0^t|\mathcal{Q}_2|{\rm d}t'
					+\int_0^t|\mathcal{Q}_3|{\rm d}t'.\label{v ee Hs-3 BDG}
				\end{align}
				The mean value theorem for $\chi_R(\cdot)$ and Assumption \ref{Assumption-1} yield that
				\begin{align*}
					\|q_7\|_{\LL_2\left(U;H^{s-3/2}\right)}
					\leq C\|v_{\e,\eta}\|_{H^{s-3/2}}f(\|u_\e\|_{H^{s}}) \left(1+\|u_\e\|_{H^s}\right).
				\end{align*}
				By \eqref{tau-e,eta,T} and Assumption \ref{Assumption-1}, we see that 
				\begin{align*}
					\|q_{8}\|_{\LL_2\left(U;H^{s-3/2}\right)}
					\leq C\|v_{\e,\eta}\|_{H^{s-3/2}}q(K),\ \ t\in[0,\tau^T_{\e,\eta,K}]\ \ \p-a.s.,
				\end{align*}
				where $q(\cdot)$ is given in Assumption \ref{Assumption-1}. Therefore we find a constant $C=C(K)>0$ such that
				\begin{align}
					\E \int_0^{\tau^T_{\e,\eta,K}}|\mathcal{Q}_{3}|{\rm d}t
					\leq&
					C(K)\E \int_0^{\tau^T_{\e,\eta,K}}\|v_{\e,\eta}\|^{2}_{H^{s-3/2}}{\rm d}t\leq C(K)\int_0^T\E\sup_{t'\in[0,\tau^t_{\e,\eta,K}]}\|v_{\e,\eta}(t')\|^{2}_{H^{s-3/2}}{\rm d}t.
				\end{align}
				Then we employ the BDG inequality to \eqref{v-ee Hs-3} to find
				\begin{align*}
					&\E\sup_{t\in[0,\tau^T_{\e,\eta,K}]}\|v_{\e,\eta}(t)\|^{2}_{H^{s-3/2}}\\
					\leq&
					C(K)\E\left(\int_0^{\tau^T_{\e,\eta,K}}\|v_{\e,\eta}\|^4_{H^{s-3/2}}
					{\rm d}t\right)^{\frac12}
					+\sum_{i=2}^3\E \int_0^{\tau^T_{\e,\eta,K}}|\mathcal{Q}_{i}|{\rm d}t\\
					\leq&
					\frac12\E\sup_{t\in[0,\tau^T_{\e,\eta,K}]}\|v_{\e,\eta}\|^{2}_{H^{s-3/2}}
					+C(K)\E\int_0^{\tau^T_{\e,\eta,K}}\|v_{\e,\eta}\|^{2}_{H^{s-3/2}}
					{\rm d}t+\sum_{i=2}^3\E \int_0^{\tau^T_{\e,\eta,K}}|\mathcal{Q}_{i}|{\rm d}t\notag\\
					\leq&
					\frac12\E\sup_{t\in[0,\tau^T_{\e,\eta,K}]}\|v_{\e,\eta}\|^{2}_{H^{s-3/2}}
					+C(K)\int_0^{T}\E\sup_{t'\in[0,\tau^t_{\e,\eta,K}]}\|v_{\e,\eta}(t')\|^{2}_{H^{s-3/2}}
					{\rm d}t+\E \int_0^{\tau^T_{\e,\eta,K}}|\mathcal{Q}_{2}|{\rm d}t,
				\end{align*}
				On account of Lemma \ref{N2 Lemma}, we arrive at
				\begin{align}
					\E \int_0^{\tau^T_{\e,\eta,K}}|\mathcal{Q}_2|{\rm d}t
					\leq&
					C(K)\E \int_0^{\tau^T_{\e,\eta,K}}\|v_{\e,\eta}\|^{2}_{H^{s-3/2}}{\rm d}t
					+C(K)T\max\{\e,\eta\}\notag\\
					\leq&
					C(K) \int_0^{T}\E\sup_{t'\in[0,\tau^t_{\e,\eta,K}]}\|v_{\e,\eta}(t')\|^{2}_{H^{s-3/2}}{\rm d}t
					+C(K)T\max\{\e,\eta\}.
				\end{align}
				Now we can put these all together to obtain
				\begin{align}
					\E\sup_{t\in[0,\tau^T_{\e,\eta,K}]}\|v_{\e,\eta}(t)\|^2_{H^{s-3/2}}
					\leq
					C(K) \int_0^{T}\E\sup_{t'\in[0,\tau^t_{\e,\eta,K}]}\|v_{\e,\eta}(t')\|^2_{H^{s-3/2}}{\rm d}t
					+C(K)T\max\{\e,\eta\},
				\end{align}
				which means that
				\begin{align}\label{Cauchy H s-3/2 ee}
					\E\sup_{t\in[0,\tau^T_{\e,\eta,K}]}\|v_{\e,\eta}(t)\|^2_{H^{s-3/2}}
					\leq
					C(K,T)\max\{\e,\eta\},
				\end{align}
				and hence \eqref{Cauchy H s-3/2} holds true. 
			\end{proof}

			\begin{Lemma}\label{Convergence of ue}
				For any fixed $s>3$ and $T>0$, there is an $\{\mathcal{F}_t\}_{t\geq0}$ progressive measurable $H^{s}$-valued process
				\begin{equation}\label{u L2 bound}
					u\in L^2\left(\Omega; L^\infty\left(0,T;H^{s}\right) \right)
				\end{equation}
				and a countable subsequence of $\{u_\e\}$ $($still denoted as $\{u_\e\})$ such that 
				\begin{equation}\label{convergence ye a.s.}
					u_\e\xrightarrow[]{\e\rightarrow 0}u \  {\rm in}\  C([0,T];H^{s-3/2})\ \ \p-a.s.
				\end{equation}
			\end{Lemma}
			
			\begin{proof}
				We notice that for each $\e\in(0,1)$, \eqref{approximate problem} has solution $u_\e$ almost surely. Now we first take $\e$ to be discrete such that for all $\e$, $u_\e$ can be defined on the same set $\widetilde{\Omega}$ with $\p\{\widetilde{\Omega}\}=1$ (Actually, one can pick  discrete $\e=\frac{1}{n}$ from the beginning in \eqref{approximate problem}).  Recall \eqref{tau-e,T} and \eqref{tau-e,eta,T}. For any $\epsilon>0$, by using Proposition \ref{global solution to appro and estimates} and Chebyshev's inequality, we see that
				\begin{align*}
					&\p\left\{\sup_{t\in[0,T]}\|u_\e-u_\eta\|_{H^{s-\frac32}}>\epsilon\right\}\\
					=\, &\p\left\{
					\left(\left\{\tau^T_{\e,\eta,K}<T\right\}\cup \left\{\tau^T_{\e,\eta,K}=T\right\}\right)\cap
					\left\{\sup_{t\in[0,T]}\|u_\e-u_\eta\|_{H^{s-\frac32}}>\epsilon\right\}\right\}\\
					\leq\, & \p\left\{\tau^T_{\e,K}<T\right\}+\p\left\{\tau^T_{\eta,K}<T\right\}
					+\p\left\{
					\sup_{t\in[0,\tau^T_{\e,\eta,K}]}\|u_\e-u_\eta\|_{H^{s-\frac32}}>\epsilon\right\}\\
					\leq\, &\frac{2C(R,T,u_0)}{K^2}
					+\p\left\{
					\sup_{t\in[0,\tau^T_{\e,\eta,K}]}\|u_\e-u_\eta\|_{H^{s-\frac32}}>\epsilon\right\}.
				\end{align*}
				Now \eqref{Cauchy H s-3/2} clearly forces
				$$\lim_{\e\rightarrow0}\sup_{\eta\leq\e}\p\left\{\sup_{t\in[0,T]}\|u_\e-u_\eta\|_{H^{s-3/2}}>\epsilon\right\}
				\leq
				\frac{2C(R,T,u_0)}{K^2},\ \ K>1.$$
				Letting $K\rightarrow\infty$, we see that $u_\e$ converges in probability in $C([0,T];H^{s-3/2})$. 
				Therefore, up to a further subsequence, \eqref{convergence ye a.s.} holds true. 
				
				Now we prove \eqref{u L2 bound}. Indeed, since $H^{s}\hookrightarrow H^{s-3/2}$ is continuous, there exist  continuous maps $\pi_m: H^{s-3/2}\to H^{s},\ m\ge 1$ such that
				$$ \|\pi_m u\|_{H^s} \le \|u\|_{H^s}\ \  
				\lim_{m\rightarrow\infty} \|\pi_m u\|_{H^s} = \|u\|_{H^s},\ \ u\in H^{s-3/2},$$  where $\|u\|_{H^s}\triangleq\infty$ if $u\notin H^s.$ For example, one may take $\pi_m$ as the  standard mollifier.
				Then it follows from Proposition \ref{global solution to appro and estimates} and  Fatou's lemma that
				\begin{align*}
					\E\sup_{t\in[0,T]}\|u(t)\|^2_{H^s}
					\leq&
					\liminf_{m\rightarrow \infty}
					\E\sup_{t\in[0,T]}\|\pi_mu(t)\|^2_{H^s}\notag\\
					\leq&
					\liminf_{m\rightarrow \infty}
					\liminf_{\e\rightarrow 0}
					\E\sup_{t\in[0,T]} \|\pi_m u_\e(t)\|^2_{H^s}\notag\\
					\leq&
					\liminf_{m\rightarrow \infty}
					\liminf_{\e\rightarrow 0}
					\E\sup_{t\in[0,T]} \|u_\e(t)\|^2_{H^s}
					<C(R,u_0,T).
				\end{align*} 	
				Hence \eqref{u L2 bound} holds true.
			\end{proof}
			
			\subsection{Global pathwise solution to the cut-off problem}\label{subsec:global:path}
			\begin{Proposition}\label{global solution to cut-off problem}
				Let $\s=(\Omega, \mathcal{F},\p,\{\mathcal{F}_t\}_{t\geq0}, \W)$ be a stochastic basis fixed in advance. Let $s>3$,  $R>1$ and $\e\in(0,1)$. Assume Assumption  \ref{Assumption-1} is satisfied. Let $u_0\in L^2(\Omega;H^{s})$ be an $H^{s}$-valued $\mathcal{F}_0$ measurable random variable. Then for  any $T>0$, \eqref{cut-off problem} has a solution $u\in  L^2\left(\Omega; C\left([0,T];H^{s}\right)\right)$.
				Moreover, there is a constant $C(R,T,u_0)>0$ such that
				\begin{equation*}
					\E\sup_{t\in[0,T]}\|u\|^2_{H^{s}}\leq C(R,T,u_0).
				\end{equation*}
			\end{Proposition}
			
			\begin{proof}
				Since for each $\e\in(0,1)$, $u_\e$ is $\{\mathcal{F}_t\}_{t\geq0}$ progressive measurable, so is $u$. 
				By Lemma \ref{Convergence of ue} and the embedding $H^{s-3/2}\hookrightarrow \Wlip$, we can send $\e\rightarrow0$ in \eqref{approximate problem} to conclude that  $u$ solves \eqref{cut-off problem}. Now we only need to prove \eqref{L2 moment bound}. 
				Due to  Lemma \ref{Convergence of ue},  $u\in C([0,T];H^{s-3/2})\cap L^\infty\left(0,T;H^s\right) $ almost surely. Since $H^s$ is dense in $H^{s-3/2}$, we see that (\cite[page 263, Lemma 1.4]{Temam-1977-book}) $u\in C_w\left([0,T];H^s\right)$,  where $C_w\left([0,T];H^s\right)$ is the space of weakly continuous functions with values in $H^s$.  Therefore to prove \eqref{L2 moment bound}, we only need to  prove the continuity of $[0,T]\ni t\mapsto \|u(t)\|_{H^s}$.  
				
				However, we cannot directly apply the It\^{o} formula for $\|u\|^2_{H^s}$ to get control of $\E\|u(t)\|_{H^s}^2$ because we only have $u\in H^s$ and $(\HH u)u_x\in H^{s-1}$. Indeed, the It\^{o} formula in a Hilbert space  (\cite[Theorem 4.32]{Prato-Zabczyk-2014-Cambridge} or \cite[Theorem 2.10]{Gawarecki-Mandrekar-2010-Springer})  requires $ \left((\HH u)u_x, u\right)_{H^s}$ to be well-defined and the It\^{o} formula  under a  Gelfand triplet (\cite[Theorem I.3.1]{Krylov-Rozovskiui-1979-chapter} or \cite[Theorem 4.2.5]{Prevot-Rockner-2007-book}) requires the dual product ${}_{H^{s-1}}\langle (\HH u) u_x, u\rangle_{H^{s+1}}$ to be well-defined. In our case neither of them is satisfied.  To this end, we recall the mollifier $J_\e$ defined in Section \ref{Preliminaries} and apply the It\^{o} formula to $\|J_\e u\|^2_{H^s}$ to obtain
				\begin{align}
					{\rm d}\|J_\e u(t)\|^2_{H^s}
					=&2\chi_R(\|u\|_{\Wlip}) \left(J_\e h(t,u){\rm d}\W,J_\e u\right)_{H^s}\nonumber\\
					&-2\chi_R(\|u\|_{\Wlip}) \left(J_{\varepsilon} [(\HH u)u_x], J_\e u\right)_{H^s}{\rm d}t\notag\\
					&+\chi^2_R(\|u\|_{\Wlip})\|J_\e h(t,u)\|_{\LL_2(\U;H^s)}^2{\rm d}t.\label{TnX 2}
				\end{align}
				By \eqref{u L2 bound}, 
				\begin{equation} 
					\tau_N=\inf\{t\ge0:\|u(t)\|_{H^s}>N\}\rightarrow \infty\ \text{as}\ N\rightarrow\infty\ \ \p-a.s.\label{tau N continuous in t}
				\end{equation}
				Then we only need to prove the continuity up to time $\tau_N\wedge T$ for each $N\ge 1$. 
				We first notice that $J_\e$ satisfies \eqref{mollifier property 4}, \eqref{H Je Ds} and \eqref{H Je Hs norm}. 
				Therefore for any $[t_2,t_1]\subset[0,T]$ with $t_1-t_2<1$, we use Lemma \ref{Huux Hs inner product}, the BDG inequality and Assumption \ref{Assumption-1}  and \eqref{tau N continuous in t} to find
				\begin{align*}
					\E\left[\left(\|J_\e u(t_1\wedge\tau_N)\|^2_{H^s}-\|J_\e u(t_2\wedge\tau_N)\|^2_{H^s}\right)^4\right]
					\leq& C(N,T)|t_1-t_2|^{2}.
				\end{align*}
				Using Fatou's lemma, we arrive at
				\begin{align*}
					\E\left[\left(\|u(t_1\wedge\tau_N)\|^2_{H^s}-\| u(t_2\wedge\tau_N)\|^2_{H^s}\right)^4\right]
					\leq& C(N,T)|t_1-t_2|^{2}.
				\end{align*}
				This and Kolmogorov's continuity theorem ensure the continuity of $t\mapsto\|u(t\wedge\tau_N)\|_{H^{s}}$, completing the proof. 
			\end{proof}

			\subsection{Concluding the proof of Theorem \ref{Local pathwise solution}}
			Finally, we are in the position to finish the proof of Theorem \ref{Local pathwise solution}.  For the sake of clarity, we split the proof into three steps.

			\textit{Step 1: Existence.} For $u_0(\omega,x)\in L^2(\Omega; H^s)$, we let  
			$$
			\Omega_k=\{k-1\leq\|u_0\|_{H^s}<k\},\ k\in\N,\ k\ge1.
			$$
			Since $\E\|u_0\|^2_{H^s}<\infty$, we have  
			\begin{align*}
				u_0(\omega,x)&=\sum_{k\geq1}u_{0,k}(\omega,x)
				=\sum_{k\geq1}u_0(\omega,x)\textbf{1}_{\{k-1\leq\|u_0\|_{H^s}<k\}}\ \  \p-a.s.
			\end{align*}
			On account of Proposition \ref{global solution to cut-off problem}, we let $u_{k,R}$ be the pathwise  global solution to the cut-off problem \eqref{cut-off problem} with initial value $u_{0,k}$ and cut-off function $\chi_R(\cdot)$. Define
			\begin{equation}\label{Remove tau k}
				\tau_{k,R}=\inf\left\{t>0:\sup_{t'\in[0,t]}\|u_{k,R}(t')\|^2_{H^s}>\|u_{0,k}\|^2_{H^s}+2\right\}.
			\end{equation}
			Then for any $R>0$ and $k\in\N$, we have $\p\{\tau_{k,R}>0\}=1$. The difficulty here is that we have to take $R$ to be deterministic. Otherwise Proposition \ref{global solution to appro and estimates} will fail. To overcome this difficulty, we let $R=R_k$ be discrete and then denote $(u_k,\tau_k)=(u_{k,R_k},\tau_{k,R_k})$. It is clear that
			$\p\{\tau_k>0,\ \forall\ k\ge1\}=1$. Let $E>0$ be the embedding constant such that $\|\cdot\|_{\Wlip}\leq E \|\cdot\|_{H^s}$ for $s>3$. Particularly, we take $R^2_k>E^2\|u_{0,k}\|^2_{H^s}+2E^2$, and then we have
			$$\p\left\{\|u_k\|^2_{W^{1,\infty}}\leq E^2\|u_k\|^2_{H^{s}}\leq E^2\|u_{0,k}\|^2_{H^{s}}+2E^2<R^2_k,\ \forall\ t\in[0,\tau_{k}],\ \forall\  k\ge1\right\}=1,$$
			which means $$\p\left\{\chi_{R_k}(\|u_k\|_{W^{1,\infty}})=1,\ \forall\ t\in[0,\tau_k],\ \forall\ k\ge1\right\}=1.$$
			Therefore $(u_k,\tau_k)$ is the pathwise solution  to \eqref{SCCF problem} with initial value $u_{0,k}$.
			Notice that
			\begin{align*}
				\textbf{1}_{\Omega_k}u_k(t\wedge \tau_k)-\textbf{1}_{\Omega_k}u_{0,k}
				=&\int_0^{t\wedge \textbf{1}_{\Omega_k}\tau_k}\textbf{1}_{\Omega_k}[(\HH u_k)\partial_x u_k]\, {\rm d}t'
				+\int_0^{t\wedge \textbf{1}_{\Omega_k}\tau_k}\textbf{1}_{\Omega_k} h(t,u_k){\rm d}\W,
			\end{align*}
			$$\textbf{1}_{\Omega_k} h(t,u_k)=h(t,\textbf{1}_{\Omega_k} u_k)-\textbf{1}_{\Omega^C_k}h(t,0)$$
			and
			$$\textbf{1}_{\Omega_k} [(\HH u_k)\partial_x u_k]=(\HH \textbf{1}_{\Omega_k} u_k)\partial_x \textbf{1}_{\Omega_k} u_k.$$
			By Assumption \ref{Assumption-1}, we have $\|h(t,\mathbf{0})\|_{\LL_2(\U;H^s)}<\infty$. Then we have
			\begin{align*}
				&\textbf{1}_{\Omega_k}u_k(t\wedge \tau_k)-\textbf{1}_{\Omega_k}u_{0,k}\\
				=&\textbf{1}_{\Omega_k}u_k(t\wedge \textbf{1}_{\Omega_k}\tau_k)-u_{0,k}\\
				=&-\int_0^{t\wedge \textbf{1}_{\Omega_k}\tau_k}(\HH \textbf{1}_{\Omega_k} u_k)\partial_x (\textbf{1}_{\Omega_k} u_k)\, {\rm d}t'
				+\int_0^{t\wedge \textbf{1}_{\Omega_k}\tau_k}h(t,\textbf{1}_{\Omega_k} u_k){\rm d}\W.
			\end{align*}
			Therefore $(\textbf{1}_{\Omega_k}u_k,\textbf{1}_{\Omega_k}\tau_k)$ is a solution to \eqref{SCCF problem} with initial data $u_{0,k}$. Since $\Omega_k\bigcap\Omega_{k'}=\emptyset$ for $k\neq k'$ and $\bigcup_{k\geq1}\Omega_k$ is a set of full measure, we  see that
			\begin{equation*}
				\left(u=\sum_{k\geq1}\textbf{1}_{\{k-1\leq\|u_0\|_{H^s}<k\}}u_k,\ \
				\tau=\sum_{k\geq1}\textbf{1}_{\{k-1\leq\|u_0\|_{H^s}<k\}}\tau_k\right)
			\end{equation*}
			is a pathwise solution to \eqref{SCCF problem} corresponding to the initial condition $u_0$. Besides, using \eqref{Remove tau k}, we have
			\begin{align*}
				\sup_{t\in[0,\tau]}\|u\|_{H^s}^2
				=&\sum_{k\geq1}\textbf{1}_{\{k-1\leq\|u_0\|_{H^s}<k\}}\sup_{t\in[0,\tau_k]}\|u_k\|_{H^s}^2\\
				\leq&\sum_{k\geq1}\textbf{1}_{\{k-1\leq\|u_0\|_{H^s}<k\}}\left( \|u_{0,k}\|^2_{H^s}+2\right)
				\leq2\|u_{0}\|^2_{H^s}+4.
			\end{align*}
			Taking expectation gives rise to \eqref{L2 moment bound}. 
			
			\textit{Step 2: Uniqueness and maximal pathwise solution.} With $(u,\tau)$ in hand, we can extend $(u,\tau)$ to a maximal pathwise solution in the sense of Definition \ref{pathwise solution definition} by following the techniques as in \cite{Crisan-Flandoli-Holm-2018-JNS,GlattHoltz-Vicol-2014-AP,GlattHoltz-Ziane-2009-ADE,Rockner-Zhu-Zhu-2014-SPTA}. 
			For uniqueness, we let $(u_1,\tau_1)$ and $(u_2,\tau_2)$ be two solutions to \eqref{SCCF problem} such that $u_j(0)=u_0$ almost surely and $u_j(\cdot\wedge \tau_j)\in L^2\left(\Omega; C\left([0,\infty);H^s\right)\right)$ with $s>3$ for $j=1,2$. Let $\frac12<\delta<s-1$ and define
			\begin{equation*}
				\tau^{T}_{K}=\inf\left\{t\geq0:\|u_1(t)\|_{H^{s}}+\|u_2(t)\|_{H^{s}}\geq K\right\}\wedge T,\ \ K\in\N,\ T>0.
			\end{equation*}
			Using \eqref{assumption 2 for h} and the definition of $\tau^{T}_{K}$, 
			then the estimate of $\E\sup_{t\in[0,\tau^T_{K}]}\|u_1(t)-u_2(t)\|^2_{H^{\delta}}$ is essential as in the derivation of \eqref{Cauchy H s-3/2} and we have
			\begin{align*}
				\E\sup_{t\in[0,\tau^T_{K}]}\|u_1(t)-u_2(t)\|^2_{H^{\delta}}
				=0.
			\end{align*}
			If necessary, to guarantee $\tau^T_{K}>0$ almost surely, we can first assume $u_0\in L^\infty(\Omega;H^s)$ and then remove this restriction by using the techniques as in \textit{Step 1}. Hence we obtain uniqueness and  the details are omitted here for brevity.  
			
			\textit{Step 3: Blow-up criterion.} We first define
			\begin{align*}
				\tau_{1,m}=\inf\left\{t\geq0: \|u(t)\|_{H^s}\geq m\right\},\ \ \
				\tau_{2,n}=\inf\left\{t\geq0: \|u_x(t)\|_{L^{\infty}}+\|\HH u_x(t)\|_{L^{\infty}}\geq n\right\},
			\end{align*}
			and then let $\displaystyle\tau_1=\lim_{m\rightarrow\infty}\tau_{1,m}$ and $\displaystyle\tau_2=\lim_{n\rightarrow\infty}\tau_{2,n}$.  
			We notice that
			for fixed $m,n>0$, even if $\p\{\tau_{1,m}=0\}$ or $\p\{\tau_{2,n}=0\}$ is larger than $0$, for a.e. $\omega\in\Omega$, there is $m>0$ or $n>0$ such that $\tau_{1,m},\tau_{2,n}>0$. By continuity of $\|u(t)\|_{H^s}$ and the uniqueness of $u$, it is easy to check that $\tau_1=\tau_2$ is actually the maximal existence time $\tau^*$ of $u$ in the sense of Definition \ref{pathwise solution definition}. Therefore
			to prove \eqref{Blow-up criterion common},   we only need to
			verify that 
			\begin{equation}\label{tau1=tau2 blow-up}
				\tau_{1}=\tau_{2} \ \ \p-a.s.
			\end{equation}
			The approach here is motivated by \cite{Crisan-Flandoli-Holm-2018-JNS,Alonso-Rohde-Tang-2021-JNLS}. Since $H^s\hookrightarrow W^{1,\infty}$ and $\HH$ is continuous in $H^s$ (cf. \eqref{H Je Hs norm}), there exists a constant $M>0$ such that,
			\begin{align*}
				\sup_{t\in[0,\tau_{1,m}]}\left(\|u_x(t)\|_{L^{\infty}}+\|\HH u_x\|_{L^{\infty}}\right)\leq M\sup_{t\in[0,\tau_{1,m}]}\|u(t)\|_{H^s}
				\leq ([M]+1)m,
			\end{align*}
			where $[M]$ denotes the integer part of $M$. Therefore we have
			$\tau_{1,m}\leq\tau_{2,([M]+1)m}\leq \tau_2$ $\p-a.s.,$
			which means that
			$\tau_{1}\leq \tau_2$ $\p-a.s.$
			Now we only need to prove $\tau_{2}\leq \tau_1$ $\p-a.s.$ We do the following claim:
			
			\textbf{Claim:} 
			\begin{align}
				\p\left\{\sup_{t\in[0,\tau_{2,n_1}\wedge n_2]}\|u(t)\|_{H^s}<\infty\right\}=1,\ \
				\forall\ n_1,n_2\in\N.\label{tau2<tau1 condition}
			\end{align}
			As is explained before,
			we cannot directly apply the It\^{o} formula for $\|u\|^2_{H^s}$ to get control of $\E\|u(t)\|_{H^s}^2$.
			Similar to \eqref{TnX 2}, by applying $J_\e$ to \eqref{SCCF problem} and using the It\^{o} formula for $\|J_\e u\|^2_{H^s}$, we have that 
			for  any $t>0$,	
			\begin{align}
				{\rm d}\|J_\e u(t)\|^2_{H^s}
				=& \left(J_\e h(t,u){\rm d}\W,J_\e u\right)_{H^s}
				-2\left(D^sJ_\e\left[(\HH u)u_x\right],D^sJ_\e u\right)_{L^2}\, {\rm d}t 
				+ \|J_\e h(t,u)\|_{\LL_2(\U;H^s)}^2\, {\rm d}t.\label{Jn X 2}
			\end{align}
			By the BDG inequality, we have
			\begin{align*}
				&\E\sup_{t\in[0,\tau_{2,n_1}\wedge n_2]}\|J_\e u(t)\|^2_{H^s}\\
				\leq&\E\|J_\e u_0\|^2_{H^s}
				+C\E\left(\int_0^{\tau_{2,n_1}\wedge n_2}\|J_\e h(t,u)\|^2_{\LL_2(\U; H^s)}\|J_\e u\|^2_{H^s}{\rm d}t\right)^{\frac12}\\
				&+2\E\int_0^{\tau_{2,n_1}\wedge n_2}|\left(D^sJ_\e\left[(\HH u)u_x\right],D^sJ_\e u\right)_{L^2}|{\rm d}t
				+\E\int_0^{\tau_{2,n_1}\wedge n_2}\|J_\e h(t,u)\|_{\LL_2(\U;H^s)}^2{\rm d}t.
			\end{align*}
			Then \eqref{assumption 1 for h} and \eqref{H Je Hs norm}  lead to
			\begin{align*}
				&C\E\left(\int_0^{\tau_{2,n_1}\wedge n_2}\|J_\e h(t,u)\|^2_{\LL_2(\U; H^s)}\|J_\e u\|^2_{H^s}{\rm d}t\right)^{\frac12}\\
				\leq& \frac12\E\sup_{t\in[0,\tau_{2,n_1}\wedge n_2]}\|J_\e u\|_{H^s}^2
				+Cf^2(2n_1)\int_0^{n_2}\left(1+\E\|u\|_{H^s}^2\right){\rm d}t.
			\end{align*}
			By Lemma \ref{Huux Hs inner product}, we find
			\begin{align}
				2\E\int_0^{\tau_{2,n_1}\wedge n_2}|\left(D^sJ_\e\left[(\HH u)u_x\right],D^sJ_\e u\right)_{L^2}|{\rm d}t\leq Cn_1\int_0^{n_2}\left(1+\E\|u\|_{H^s}^2\right){\rm d}t.\label{Huux u blow-up}
			\end{align}
			It follows from \eqref{assumption 1 for h} that for some constant $C>0$,
			\begin{align*}
				\E\int_0^{\tau_{2,n_1}\wedge n_2}\|J_\e h(t,u)\|_{\LL_2(\U;H^s)}^2\, {\rm d}t\leq Cf^2(2n_1)\int_0^{n_2}\left(1+\E\|u\|_{H^s}^2\right){\rm d}t,
			\end{align*}
			Therefore we combine the above estimates, use \eqref{H Je Hs norm},
			and then send $\e\rightarrow0$ in the resulting inequality to obtain
			\begin{align}\label{u Hs estimate}
				\E\sup_{t\in[0,\tau_{2,n_1}\wedge n_2]}\|u(t)\|^2_{H^s}
				\leq C\E\|u_0\|^2_{H^s}+ C\int_0^{n_2}\left(1+\E\sup_{t'\in[0,t\wedge\tau_{2,n_1}]}\|u(t')\|_{H^s}^2\right){\rm d}t.
			\end{align}
			Then Gr\"{o}nwall's inequality shows that for each $n_1,n_2\in\N$, there is a constant $C=C(n_1,n_2,u_0)>0$ such that $$\E\sup_{t\in[0,\tau_{2,n_1}\wedge n_2]}\|u(t)\|^2_{H^s}<C(n_1,n_2,u_0),$$ which gives \eqref{tau2<tau1 condition} and concludes the claim.
			
			Hence \eqref{tau2<tau1 condition} implies that for all $n_1,n_2\in\N$, $\p\left\{\sup_{t\in[0,\tau_{2,n_1}\wedge n_2]}\|u(t)\|_{H^s}<\infty\right\}=1$. On the other hand,
			it is easy to see that for all $n_1,n_2\in\N$,
			\begin{align*}
				\left\{\sup_{t\in[0,\tau_{2,n_1}\wedge n_2]}\|u(t)\|_{H^s}<\infty\right\}
				\subset\bigcup_{m\in\N}\left\{\tau_{2,n_1}\wedge n_2\leq\tau_{1,m}\right\}
				\subset\left\{\tau_{2,n_1}\wedge n_2\leq\tau_{1}\right\}.
			\end{align*}
			Consequently,
			\begin{align}
				\p\left\{\tau_2\leq\tau_1\right\}
				=\p\left\{\bigcap_{n_1\in\N}\left\{\tau_{2,n_1}\leq \tau_{1}\right\}\right\}
				=\p\left\{\bigcap_{n_1,n_2\in\N}\left\{\tau_{2,n_1}\wedge n_2\leq \tau_{1}\right\}\right\}=1.\label{tau2<tau1}
			\end{align}
			Combining the above three steps, we complete the proof of Theorem \ref{Local pathwise solution}.	
			


			\section{Proof of Theorem \ref{Non blowup}}\label{Section:global:strong}
			
			To begin with, we can follow the steps as in the proof of Theorem \ref{Local pathwise solution} to obtain that, if $u_0$ is an $H^s$-valued $\mathcal{F}_0$-measurable random variable satisfying $\E\|u_0\|^2_{H^s}<\infty$ with $s>3$, then \eqref{SCCF non blow up Eq} has a unique pathwise solution $u\in H^s$ with maximal existence time $\tau^*$. 
			Now the target is to show that $\p\{\tau^*=\infty\}=1$. To this end, we define
			\begin{align*}
				\widehat{\tau}_{k}=\inf\left\{t\geq0: \|u(t)\|_{H^{s-1}}\geq k\right\},\ \ k\ge1\ \ \text{and}\ \ \widehat{\tau^*}=\lim_{k\rightarrow\infty}\widehat{\tau}_{k}.
			\end{align*}
			Recalling Remark \ref{Blow-up time remark}, we have 
			\begin{equation}\label{hat tau=tau star}
				\widehat{\tau^*}=\tau^*\ \ \p-a.s.
			\end{equation}
				Therefore we only need to show $\p\{\widehat{\tau^*}=\infty\}=1$.
			Applying the It\^{o} formula to 
			$\|u(t)\|^2_{H^{s-1}}$ gives
			\begin{align}
				{\rm d}\| u\|^2_{H^{s-1}}
				=\, & 2\left( \alpha(t,u), u\right)_{H^{s-1}}\, {\rm d}W
				-2\left((\HH u)u_x,  u\right)_{H^{s-1}}\, {\rm d}t
				+\| \alpha(t,u)\|^2_{H^{s-1}}\, {\rm d}t.\label{Jn non blow-up}
			\end{align}
			Let $\G\in\mathfrak{G}$. On account of the It\^{o} formula, we derive
			\begin{align*}
				{\rm d}\G(\| u\|^2_{H^{s-1}})
				=\, &2\G'(\| u\|^2_{H^{s-1}})\left( \alpha(t,u),  u\right)_{H^{s-1}}\, {\rm d}W\notag\\
				&+\G'(\| u\|^2_{H^{s-1}})
				\left\{-2\left((\HH u)u_x,  u\right)_{H^{s-1}}
				+\|\alpha(t,u)\|^2_{H^{s-1}}\right\}\, {\rm d}t\notag\\
				&	+2\G''(\| u\|^2_{H^{s-1}})\left|\left(\alpha(t,u), u\right)_{H^{s-1}}\right|^2\, {\rm d}t.
			\end{align*}
		Recall that in Assumption \ref{Assumption-alpha},
			$$\mathcal{M}(t)=2Q(\|u_x(t)\|_{L^\infty}+\|\HH u_x(t)\|_{L^\infty})\|u(t)\|^2_{H^{s-1}}
			+\|\alpha(t,u)\|^2_{H^{s-1}}.$$
			Hence taking expectation, using inequality \eqref{H Je Hs norm},  Lemma \ref{Huux Hs inner product} and Assumption \ref{Assumption-alpha} we find that for any $t>0$,
			\begin{align}
				&\E \G(\| u(t)\|^2_{H^{s-1}})  \notag\\
				=\, &\E \G(\|u_0\|^2_{H^{s-1}})+\E\int_0^{t}\G'(\|u\|^2_{H^{s-1}})
				\left\{-2\left( (\HH u)u_x , u\right)_{H^{s-1}}
				+\|\alpha(t',u)\|^2_{H^{s-1}}\right\}\, {\rm d}t'\notag\\
				&
				+\E\int_0^{t}2\G''(\|u\|^2_{H^{s-1}})\left|\left(\alpha(t',u), u\right)_{H^{s-1}}\right|^2\, {\rm d}t'\notag\\
				\leq\, &\E \left\{\G(\|u_0\|^2_{H^{s-1}})+\E\int_0^{t}\G'(\|u\|^2_{H^{s-1}})
				\mathcal{M}(t')+2\G''(\|u\|^2_{H^{s-1}})\left|\left(\alpha(t',u), u\right)_{H^{s-1}}\right|^2\right\}\, {\rm d}t'\notag\\
				\leq\, & \E \G(\| u_0\|^2_{H^{s-1}})+K_1 t-\E\int_0^{t}
				K_2\frac{\left\{\G'(\|u\|^2_{H^{s-1}})\left|\left(\alpha(t',u), u\right)_{H^{s-1}}\right|\right\}^2}{ 1+\G(\|u\|^2_{H^{s-1}})}
				\, {\rm d}t',\label{Huux u global existence}
			\end{align}
			which shows that there exists a constant $C=C(u_0,K_1,K_2,t)>0$ such that
			\begin{align}
				\E\int_0^{t}
				\frac{\left\{\G'(\|u\|^2_{H^{s-1}})\left|\left(\alpha(t',u), u\right)_{H^{s-1}}\right|\right\}^2}
				{ 1+\G(\|u\|^2_{H^{s-1}})}\, {\rm d}t'
				\leq C(u_0,K_1,K_2,t).\label{to use BDG 2}
			\end{align}
			Moreover,  for any $T>0$, it follows from Assumption \ref{Assumption-alpha} and the BDG inequality that
			\begin{align*}
				&\E\sup_{t\in[0,{T}]}\G(\|u\|^2_{H^{s-1}})\\
				\leq\, & \E \G(\|u_0\|^2_{H^{s-1}})+
				C\E\left(\int_0^{T}
				\left\{\G'(\|u\|^2_{H^{s-1}})\left|\left( \alpha(t,u), u\right)_{H^{s-1}}\right|\right\}^2
				\, {\rm d}t\right)^\frac12\\
				&+K_1T+K_2 \E\int_0^{T}
				\frac{\left\{\G'(\|u\|^2_{H^{s-1}})\left|\left(\alpha(t,u), u\right)_{H^{s-1}}\right|\right\}^2}{ 1+\G(\|u\|^2_{H^{s-1}})} \, {\rm d}t\\
				\leq\, &\E \G(\|u_0\|^2_{H^s})+
				\frac12\E\sup_{t\in[0,{T}]}\left(1+\G(\|u\|^2_{H^{s-1}})\right)
				+C\E\int_0^{T}
				\frac{\left\{\G'(\|u\|^2_{H^{s-1}})
					\left|\left(\alpha(t,u), u\right)_{H^{s-1}}\right|\right\}^2}
				{ 1+\G(\|u\|^2_{H^{s-1}})}
				\, {\rm d}t\\
				&+K_1T
				+K_2\E\int_0^{T}\frac{\left\{\G'(\|u\|^2_{H^{s-1}})\left|\left(\alpha(t,u), u\right)_{H^{s-1}}\right|\right\}^2}{ 1+\G(\|u\|^2_{H^{s-1}})}
				\, {\rm d}t.
			\end{align*}
			Thus, using \eqref{to use BDG 2} we obtain
			\begin{align*}
				\E\sup_{t\in[0,{T}]}\G(\| u\|^2_{H^{s-1}})
				&\leq\, 
				C(u_0,K_1,T)+C(K_2) \E\int_0^{T}
				\frac{\left\{\G'(\|u\|^2_{H^{s-1}})\left|\left(\alpha(t,u), u\right)_{H^{s-1}}\right|\right\}^2}{ 1+\G(\|u\|^2_{H^{s-1}})}
				\, {\rm d}t\\
				&\leq\, C(u_0,K_1,K_2,T).
			\end{align*}
			We can infer from the above estimate that
			\begin{equation*}
				\p\{\widehat{\tau^*}<T\}\leq \p\{\widehat{\tau}_k<T\}\leq 
				\p\left\{\sup_{t\in[0,T]}\G(\|u\|^2_{H^{s-1}})\geq \G(k^2)\right\}\leq \frac{C(u_0,K_1,K_2,T)}{\G(k^2)}.
			\end{equation*}
			Therefore, since $\lim_{x\rightarrow\infty}\G(x)=\infty$,
			one can send $k\rightarrow\infty$ to identify that
			$\p\{\widehat{\tau^*}<T\}=0$. Since $T>0$ is arbitrary, we have that $\p\{\widehat{\tau^*}=\infty\}=1$ which shows the desired assertion.

		\begin{Remark}\label{Remark global existence}
		We remark that the using of Lyapunov is motivated by the non-explosion text \cite{Hasminskii-1969-Book}, see also \cite{Brzezniak-etal-2005-PTRF,Ren-Tang-Wang-2020-Arxiv,Rohde-Tang-2021-NoDEA}. In the above proof, \eqref{Huux u Hs-1} is used to obtain \eqref{Huux u global existence}, but we remark that \eqref{Huux u Hs} is also used implicitly. Indeed, as in the proof of \eqref{tau1=tau2 blow-up}, \eqref{Huux u Hs} is used to obtain \eqref{Huux u blow-up}, and here \eqref{hat tau=tau star}  also requires \eqref{Huux u Hs} because it is a consequence of \eqref{tau1=tau2 blow-up}. In this work we estimate $H^{s-1}$ norm (i.e., $\E\sup_{t\in[0,{T}]}\G(\| u\|^2_{H^{s-1}})$) and use the fact \eqref{hat tau=tau star} to prove global existence. Let us stress that this is different from the recent work \cite{Rohde-Tang-2021-NoDEA}, where the authors estimate $H^s$ norm of the solution to show global existence.
		\end{Remark}

			\section{Proof of Theorem \ref{linear:theorem:blowup}}\label{Sec:blow:up:linear}
			
			Due to the linear nature of the noise we use the Girsanov type transformation
			\begin{equation}\label{Girsanov:type}
				v=\frac{1}{\beta(\omega,t)} u,\ \
				\beta(\omega,t)={\rm e}^{\int_0^tb(t') {\rm d} W_{t'}-\int_0^t\frac{b^2(t')}{2} {\rm d}t'}.
			\end{equation}
			In the following lemma we show that the defined process $v$ is the solution to a random PDE enjoying desired regularity properties.
			\begin{Lemma}\label{WP:gisanov:transformation}
				Let $s>3$, $b(t)$ satisfies assumption \ref{Assumption-3} and fix $\s=\left(\Omega, \mathcal{F},\p,\{\mathcal{F}_t\}_{t\geq0}, W\right)$. Assume that $u_0(\omega,x)$ is an $H^s$-valued $\mathcal{F}_0$-measurable random variable with $\E\|u_0\|^2_{H^s}<\infty$ and $(u,\tau^*)$ is the corresponding unique maximal pathwise solution to \eqref{linear:SCCF:problem}.
				Then for $t\in[0,\tau^*)$,  the process $v$ as in \eqref{Girsanov:type} is a solution $\mathbb{P}-a.s.$ to 
				\begin{equation} \label{Solution:v:girsanov}
					\left\{\begin{aligned}
						&v_t+\beta   (\HH v) v_x=0, \quad \beta(\omega,t)= {\rm e}^{\int_0^tb(t') {\rm d} W_{t'}-\int_0^t\frac{b^2(t')}{2} {\rm d}t'}, \\
						&v(\omega,0,x)=u_0(\omega,x), \quad x\in \mathbb{R}
					\end{aligned} \right.
				\end{equation}
				with $v\in C\left([0,\tau^*);H^s\right)\bigcap C^1([0,\tau^*) ;H^{s-1})$ $\mathbb{P}-a.s.$
			\end{Lemma}
			
			\begin{proof}
				Applying Theorem \ref{Local pathwise solution} for the particular case $h(t,u)= b(t)u$ and noticing that  $b(t)$ satisfies Assumption \ref{Assumption-3} (and therefore $h(t,u)=b(t)u$ satisfies Assumption \ref{Assumption-1}), we infer that equation \eqref{linear:SCCF:problem} has a unique maximal pathwise solution $(u,\tau^*)$. It\^{o}'s formula yields 
				$$ {\rm d}\frac{1}{\beta}= -b(t)\frac{1}{\beta} {\rm d}W+ b^{2}(t)\frac{1}{\beta} {\rm d}t,$$
				and hence straightforward computation shows that 
				\begin{equation}
					{\rm d}v=-\beta   (\HH v) v_x {\rm d} t,
				\end{equation}
				yielding the first equation in \eqref{Solution:v:girsanov}. At time $t=0$, $v(\omega,0,x)=u_{0}(\omega,x)$ since $\beta(\omega,0)=1$ almost surely, so $v$ satisfies \eqref{Solution:v:girsanov} almost surely. Furthermore, Theorem \ref{Local pathwise solution} shows that $u\in C\left([0,\tau^*);H^s\right)$ $\p-a.s.$, hence $v\in C\left([0,\tau^*);H^s\right)$ and $v\in C^{1}\left([0,\tau^*);H^{s-1}\right)$ $\p-a.s.$.
			\end{proof}

				Invoking Lemma \ref{WP:gisanov:transformation} we have that a.e. $\omega\in\Omega$ the process $v(\omega,t,x)$ solves \eqref{Solution:v:girsanov} on $[0,\tau^*)$ and $v\in C\left([0,\tau^*);H^s\right)\bigcap C^1([0,\tau^*) ;H^{s-1})$ for $s>3$. In particular, by the Sobolev embedding, $v\in C\left([0,\tau^*);C^1\right)$, therefore for a.e. $\omega\in\Omega$, the particle trajectory mapping related to the process $v$ given by
				\begin{equation} \label{particle:trajectory}
					\left\{\begin{aligned}
						&\frac{{\rm d} \phi(\omega,t,x)}{{\rm d}t}=\beta(\omega,t)\HH v(\omega,t,\phi(\omega,t,x)), \quad t\in [0,\tau^*), \\
						&\phi(\omega,0,x)=x, \quad x\in \mathbb{R},
					\end{aligned} \right.
				\end{equation}
				has a unique  solution  $\phi(\omega,t,x)\in C^1([0,\tau^*)\times \mathbb{R})$.  Now for a.e. $\omega\in \Omega$, we let $x_0=x_0(\omega)\in \mathbb{R}$ be the point that $u_{0}$ attains it global maximum, i.e., 
				$$v(\omega,0,x_0(\omega))= u_{0}(\omega,x_0(\omega))=\displaystyle \max_{x\in \mathbb{R}}u_{0}(\omega,x), \mbox{ for a.e. } \omega\in \Omega. $$ Then we focus on the  particle trajectory mapping from $x_0$ in \eqref{particle:trajectory}, i.e.,
				\begin{equation} \label{particle:trajectory:x0}
					\left\{\begin{aligned}
						&\frac{{\rm d} \phi(\omega,t,x_0)}{{\rm d}t}=\beta(\omega,t)\HH v(\omega,t,\phi(\omega,t,x_0)), \quad t\in [0,\tau^*), \\
						&\phi(\omega,0,x_{0})=x_{0}.
					\end{aligned} \right.
				\end{equation}
				On the other hand, by the transport nature of equation \eqref{Solution:v:girsanov} the value of $v$ is constant along characteristics and since $v(\omega,0,x_0)$ attains a global maximum, we have that
				\begin{equation}\label{maxmium transport}
					\partial_{x}v(\omega,t,\phi(\omega,t,x_0))=0,\ \ t\in [0,\tau^*)\ \ \p-a.s.
				\end{equation}
				Computing the quantity $\Lambda v(\omega,t,\phi(\omega,t,x_0))$, i.e., the evolution of the $\Lambda$ operator of $v$ along the trajectory,  we  see that for a.e. $\omega\in\Omega$ and $t\in [0,\tau^*)$, there holds
				\begin{align}
					&\frac{{\rm d} \Lambda v(\omega,t,\phi(\omega,t,x_0))}{{\rm d}t}\notag\\
					=\,& \Lambda v_{t} (\omega,t,\phi(\omega,t,x_0)) + 
					\frac{{\rm d} \Lambda v (\omega,t,\phi(\omega,t,x_0))}{{\rm d} \phi(\omega,t,x_0))}
					\frac{{\rm d} \phi(\omega,t,x_{0})}{{\rm d}t}  \notag\\
					=\,& -\beta(\omega,t) \left\{\Lambda[(\HH v) v_{x}](\omega,t,\phi(\omega,t,x_0))- (\HH v)(\omega,t,\phi(\omega,t,x_0))\Lambda v_{x}(\omega,t,\phi(\omega,t,x_0)) \right\}, \label{chain:rule:computation}
				\end{align}
				where chain rule is used in the former equality and the fact that $v$ solves \eqref{Solution:v:girsanov} and the particle trajectory equation \eqref{particle:trajectory:x0} in the latter. Let $\tilde{v}= \HH v$. Then we have $v_{x}=\Lambda \tilde{v}$ and $\Lambda v_{x}= -\Tilde{v}_{xx}$ and then we can rewrite equation \eqref{chain:rule:computation} as
				\begin{align}
					&\frac{{\rm d} \Lambda v(\omega,t,\phi(\omega,t,x_0))}{{\rm d}t}\notag\\
					=&-\beta(\omega,t) \left[ \Lambda(\tilde{v}\Lambda\tilde{v})(\omega,t,\phi(\omega,t,x_0))+ \tilde{v}(\omega,t,\phi(\omega,t,x_0))\tilde{v}_{xx}(\omega,t,\phi(\omega,t,x_0))  \right].\label{prepare:ineq:ode}
				\end{align}
				Now we denote $z_0\triangleq\phi(\omega,t,x_0)$ and omit the dependence of $\omega$ and $t$ in \eqref{prepare:ineq:ode} for simplicity if there is no ambiguity.

\subsection{An identity for the fractional Laplacian $\Lambda$}
	\begin{Lemma}\label{identity:SV Lemma}
	For a.e. $\omega\in\Omega$ and $t\in [0,\tau^*)$, there holds the following equation
	\begin{align}\label{identity:SV}
		\Lambda(\tilde{v}\Lambda\tilde{v})(z_0)+ \tilde{v}(z_0)\tilde{v}_{xx}(z_0)  = -\frac{1}{2}(\Lambda v(z_0))^2-\frac{1}{\pi}\norm{\frac{\HH v(z_0)-\HH v(\cdot)}{z_0-\cdot}}^{2}_{\dot{H}^{\frac{1}{2}}}.
	\end{align}	
	\end{Lemma}
\begin{proof}
		We remark that \eqref{identity:SV} has been obtained in \cite[Proposition 3.5]{Silvestre-Vicol-TAMS} in the deterministic case. However, notice that one cannot assume without loss of generality (as in the deterministic case) that $\tilde{v}(z_0) = 0$ which simplifies the proof of \eqref{identity:SV} (see Remark \ref{remark blow-up}) and hence we present also the complete proof here.
	
	Recalling \eqref{maxmium transport} and $v_{x}=\Lambda \tilde{v}$, we have $\Lambda \tilde{v}(z_{0})=0$ for $t\in [0,\tau^*)$ almost surely. Then, invoking the integral representation  \eqref{frac:integral:rep} for $\alpha=1$, we  arrive at
	\begin{align}\label{equality:point:1}
		\Lambda(\tilde{v}\Lambda\tilde{v})(z_0)=& \frac{1}{\pi}\mbox{p.v.} \int_{\mathbb{R}}\frac{\tilde{v}(z_0)\Lambda\tilde{v}(z_0)-\tilde{v}(y)\Lambda\tilde{v}(y) }{\abs{z_{0}-y}^{2}}\, {\rm d}y\notag\\
		=& \frac{1}{\pi}\mbox{p.v.} \int_{\mathbb{R}}\frac{-\tilde{v}(y)\Lambda\tilde{v}(y)}{\abs{z_{0}-y}^{2}} \, {\rm d}y,\ \ t\in [0,\tau^*)\ \ \p-a.s.
	\end{align}
	On the other hand, we have that
	\[ \tilde{v}(z_0)\tilde{v}_{xx}(z_0)=-\tilde{v}(z_{0})\Lambda(\Lambda \tilde{v})(z_{0})= \tilde{v}(z_0)\frac{1}{\pi} \mbox{p.v.}\int_{\mathbb{R}}\frac{\Lambda \tilde{v}(y)}{\abs{z_{0}-y}^{2}}\,{\rm d}y, \ \ t\in [0,\tau^*)\ \ \p-a.s., \]
	where we have used in the first equality the fact that $\partial_{xx}=-\Lambda^{2}$ and the semigroup property of the fractional Laplace operator $\Lambda^{2}=\Lambda \Lambda$. Therefore, we have that
	\begin{equation}\label{v-tilde prepare}
		\Lambda(\tilde{v}\Lambda\tilde{v})(z_0)+ \tilde{v}(z_0)\tilde{v}_{xx}(z_0) =\frac{1}{\pi} \mbox{p.v.}\int_{\mathbb{R}} \frac{\left(\tilde{v}(z_0)-\tilde{v}(y)\right)\Lambda \tilde{v}(y)}{\abs{z_{0}-y}^{2}}\,{\rm d}y, \ \ t\in [0,\tau^*)\ \ \p-a.s.
	\end{equation}
	Recalling the notation $\tilde{v}(z_0)=\tilde{v}(\omega,t,z_{0})$, then for a.e. $\omega\in\Omega$ and $t\in [0,\tau^*)$ we define
	\begin{equation}\label{v-bar v-tilde}
		\overline{v}(\omega,t,y)\triangleq\tilde{v}(\omega,t,z_{0})-\tilde{v}(\omega,t,y).
	\end{equation}
	Since $\overline{v}(\omega,t,z_{0})=0$, 
	factorizing the root implies that  there exists a process $\eta(\omega,t,y)$ such that
	\begin{equation}\label{v bar eta}
		\overline{v}(\omega,t,y)=(z_{0}-y)\eta(\omega,t,y), \ \ t\in [0,\tau^*)\ \ \p-a.s.
	\end{equation}  Therefore we can observe that
	\begin{equation}\label{v eta}
		\eta(\omega,t,z_0)=-\overline{v}_{x}(\omega,t,z_0),\ \HH \eta(\omega,t,z_0)=\Lambda\overline{v}(\omega,t,z_0), \ \ t\in [0,\tau^*)\ \ \p-a.s.
	\end{equation}
	Again, we drop $\omega$ and $t$ if there is no ambiguity. 
	Then \eqref{v-tilde prepare} reduces to
	\begin{align*}
		\Lambda(\tilde{v}\Lambda\tilde{v})(z_0)+ \tilde{v}(z_0)\tilde{v}_{xx}(z_0) =&-\frac{1}{\pi} \mbox{p.v.}\int_{\mathbb{R}} \frac{\overline{v}(y)\Lambda \overline{v}(y)}{\abs{z_{0}-y}^{2}}\,{\rm d}y\\
		=&-\frac{1}{\pi}\mbox{p.v.} \int_{\mathbb{R}}
		\frac{(z_{0}-y)\eta(y)[\Lambda((z_{0}-\cdot)\eta(\cdot))](y)}{\abs{z_{0}-y}^{2}}\, {\rm d}y, \ \ t\in [0,\tau^*)\ \ \p-a.s.
	\end{align*}
	Using the linearity of the fractional Laplacian operator and Lemma \ref{hilbert:iden:3}, we have
	$$[\Lambda((z_{0}-\cdot)\eta(\cdot))](y)=z_{0}\Lambda \eta(y)-y\Lambda \eta(y)+\mathcal{H}\eta(y) = (z_{0}-y)\Lambda \eta(y)+\mathcal{H} \eta(y)$$ 
	and hence
	\begin{align*}
		\Lambda(\tilde{v}\Lambda\tilde{v})(z_0)&+\tilde{v}(z_0)\tilde{v}_{xx}(z_0)=\frac{1}{\pi} \mbox{p.v.}\int_{\mathbb{R}}\frac{-(z_{0}-y)\eta(y)\left((z_{0}-y)\Lambda \eta(y)+\mathcal{H} \eta(y) \right)}{\abs{z_{0}-y}^{2}}\, {\rm d}y \notag\\
		&=\frac{1}{\pi}\mbox{p.v.}\int_{\mathbb{R}}\frac{-(z_{0}-y)^2\eta(y)\Lambda \eta(y) \, {\rm d}y}{\abs{z_{0}-y}^{2}} -\frac{1}{\pi} \mbox{p.v.}\int_{\mathbb{R}}\frac{(z_{0}-y)\eta(y)\HH \eta(y) }{\abs{z_{0}-y}^{2}}\, {\rm d}y \notag\\
		&=-\frac{1}{\pi}\mbox{p.v.}\int_{\mathbb{R}}\eta(y)\Lambda \eta(y) \, {\rm d}y + \frac{1}{\pi}\mbox{p.v.}\int_{\mathbb{R}}\frac{\eta(y)\HH \eta(y) }{y-z_{0}}\, {\rm d}y\notag\\
		&= - \frac{1}{\pi}\norm{\Lambda^{1/2}\eta}^2_{L^2}+[\HH(\eta\HH \eta)](z_{0}), \ \ t\in [0,\tau^*)\ \ \p-a.s.
	\end{align*}
	Therefore, applying the identity \eqref{hilbert:iden:2} we can rewrite the above equation as
	\begin{equation*}
		\Lambda(\tilde{v}\Lambda\tilde{v})(z_0)+\tilde{v}(z_0)\tilde{v}_{xx}(z_0)= -\frac{1}{2}(\eta(z_{0}))^2+\frac{1}{2}(\HH \eta(z_{0}))^2-  \frac{1}{\pi}\norm{\Lambda^{1/2}\eta}^2_{L^2}, \ \ t\in [0,\tau^*)\ \ \p-a.s.
	\end{equation*}
	Using \eqref{v-bar v-tilde}, \eqref{v bar eta}, \eqref{v eta} and noticing $\Lambda \tilde{v}(z_{0})=0$, we arrive at
	\begin{align*}
		\Lambda(\tilde{v}\Lambda\tilde{v})(z_0)+\tilde{v}(z_0)\tilde{v}_{xx}(z_0)&=- \frac{1}{2}(\overline{v}_x(z_{0}))^2+\frac{1}{2}(\Lambda \overline{v}(z_{0}))^2-\frac{1}{\pi}\norm{\frac{\overline{v}(\cdot)}{z_{0}-\cdot}}^2_{\dot{H}^{\frac{1}{2}}}\\
		&=-\frac{1}{2}\left(\tilde{v}_x(z_{0})\right)^{2}-\frac{1}{\pi}\norm{\frac{\tilde{v}(z_0)-\tilde{v}(\cdot)}{z_{0}-\cdot}}^2_{\dot{H}^{\frac{1}{2}}}, \ \ t\in [0,\tau^*)\ \ \p-a.s.
	\end{align*}
	Then \eqref{identity:SV} is a direct consequence of the above equation and the fact $\tilde{v}= \HH v$. 
\end{proof}

		\subsection{Proof of Theorem \ref{linear:theorem:blowup}}
		Now we are in the position to prove Theorem \ref{linear:theorem:blowup}. To begin with, 
	we can  infer  from Lemma \ref{identity:SV Lemma} and \eqref{prepare:ineq:ode} that
		\begin{align*}
			\frac{{\rm d} \Lambda v (z_0)}{{\rm d}t} = \beta \left[ \frac{1}{2}(\Lambda v(z_0))^2+\frac{1}{\pi}\norm{\frac{\HH v(z_0)-\HH v(\cdot)}{z_0-\cdot}}^{2}_{\dot{H}^{\frac{1}{2}}} \right] \nonumber \geq \beta  \frac{1}{2}(\Lambda v(z_0))^2,\ \ t\in [0,\tau^*)\ \ \p-a.s.
		\end{align*}
		Let $F(\omega,t)\triangleq\Lambda v(\omega,t,z_0)$. Then the above estimate becomes
		\begin{equation}\label{ODE:2}
			\frac{{\rm d}F(\omega,t)}{{\rm d}t} \geq \frac{1}{2} \beta(\omega,t) F^2(\omega,t),\ \ t\in [0,\tau^*)\ \ \p-a.s.
		\end{equation}
		Let $0<K<1$ and define $\Omega^*\triangleq\{ \omega : \beta(t)\geq K e^{-\frac{b^*}{2}t} \mbox{ for all } t\}.$ If $F(\omega,0) > \frac{b^*}{K}$ almost surely, then $\tau^*<\infty$ for a.e. $\omega\in \Omega^*$. Indeed, integrating  \eqref{ODE:2}  leads to
		\begin{equation*}
			-\frac{1}{F(\omega,t)}+\frac{1}{F(\omega,0)} \geq \frac{1}{2}\int_{0}^{t} \beta(\omega,t') \, {\rm d}t', \ \ t\in [0,\tau^*)\ \ \p-a.s.
		\end{equation*}
		Since $F(\omega,0) > \frac{b^*}{K}$ almost surely, \eqref{ODE:2} means that $F$ is increasing almost surely. Therefore we restrict the above inequality to $\omega\in \Omega^*$ and we arrive at
		$$ \frac{1}{F(\omega,0)}\geq  \frac{1}{2} K  \int_{0}^{\tau^*} e^{-\frac{b^*}{2}t'} dt'= \frac{K}{b^*} \left(1-e^{-\frac{b^*}{2}\tau^*} \right)$$
		and hence
		$$ \frac{1}{F(\omega,0)}-\frac{K}{b^*} \left( 1-e^{-\frac{b^*}{2}\tau^*} \right)\geq0.$$ 
		By the assumption $F(\omega,0)>\frac{b^*}{K}$ almost surely, we arrive at
		$$ \frac{K}{b^*}e^{-\frac{b^*}{2}\tau^*} \geq \frac{K}{b^*}-\frac{1}{F(0)}>0, \quad \mbox{ a.e } \omega\in \Omega^*.$$
		Thus $\tau^*< \infty$ a.e. on $\Omega^*$ as desired. Recalling that  $\beta(\omega,t)={\rm e}^{\int_0^tb(t') {\rm d} W_{t'}-\int_0^t\frac{b^2(t')}{2} {\rm d}t'}
		$, we have shown that 
		$$ \mathbb{P}\{ \tau^* < \infty \} \geq \mathbb{P}\{\beta(t)\geq K e^{-\frac{b^*}{2}t} \mbox{ for all } t \}$$
		which implies that 
		$$ \mathbb{P}\{ \tau^* < \infty \} \geq \mathbb{P}\{{\rm e}^{\int_{0}^{t} b(t') \rm{d}W_{t'}}> K \mbox{ for all } t \}$$
		since $b^2(t)<b^*$ for all $t>0$. The proof is now complete.

			\section{Proof of Theorem \ref{Weak instability}}\label{sect:weak instability}
			In this section, we provide the proof of Theorem \ref{Weak instability}. As is mentioned in Remark \ref{remark on weak instability}, since we cannot  get an explicit expression of the solution to \eqref{SCCF problem}, we start with constructing some approximate solutions from which \eqref{sup sin t} can be established.
			Similarly as before we divide the proof into several subsections.

			\subsection{Approximate solutions and actual solutions}\label{subsec:approximate:weak}
			Following  \cite{Miao-Rohde-Tang-2021-arXiv}, we construct the approximate solution as follows. First, we fix two functions $\phi$, $\tilde{\phi}\in C_c^{\infty}$ such that
			\begin{equation}\label{phi phi-tilde}
				\phi(x)=\left\{\begin{aligned}
					&1,\ \text{if}\ |x|<1,\\
					&0,\ \text{if}\ |x|\ge 2,
				\end{aligned}\right.
				\quad\text{and}\quad \tilde{\phi}(x)=1\ \text{if}\ x\in \text{supp}\ \phi.
			\end{equation}	
			Next, we construct the following sequence of approximate solutions
			\begin{align}\label{approximation solution u-m n}
				u_{m,n}=u_{h}+u_{l},\ \ m\in\{-1,1\},
			\end{align}
			where 
			\begin{itemize}
				\item $u_h = u_{h,m,n}$ is the high-frequency part  defined by
				\begin{align}\label{high-frequency approximation solutions}
					u_h = u_{h,m,n}(t,x)= n^{-\frac{\delta}{2}-s}\phi\left(\frac{x}{n^{\delta}}\right)\cos(nx-mt),\ \ n\in\N.
				\end{align}
				\item $u_l=u_{l,m,n}$ is the low-frequency part  defined as the solution to the following  problem:
				\begin{equation}\label{low-frequency equation}
					\left\{\begin{aligned}
						&\partial_tu_l+(\HH u_l)\partial_xu_l=0,\quad x\in\R, \ t>0,\\
						&u_l(0,x)=-\HH\left(mn^{-1}\tilde{\phi}\left(\frac{x}{n^{\delta}}\right)\right),\quad  x\in\R.
					\end{aligned}
					\right.
				\end{equation}
			\end{itemize}
			In \eqref{high-frequency approximation solutions}-\eqref{low-frequency equation}, $\delta>0$ is a parameter that will be determined later in the proof. 
			
			Let us consider the problem \eqref{SCCF problem} with the initial data $u_{m,n}(0,x)$, i.e.,
			\begin{equation} \label{SCCF m n}
				\left\{
				\begin{aligned}
					&{\rm d}u+(\HH u)u_x{\rm d}t=h(t,u){\rm d}\W,\quad x\in\R, \ t>0,\\
					&u(0,x)=-\HH\left(mn^{-1}\tilde{\phi}\left(\frac{x}{n^{\delta}}\right)\right)+n^{-\frac{\delta}{2}-s}\phi\left(\frac{x}{n^{\delta}}\right)\cos(nx),\quad x\in\R.
				\end{aligned} 
				\right.
			\end{equation}
			Since Assumption \ref{Assumption-2} implies Assumption \ref{Assumption-1}, Theorem \ref{Local pathwise solution} immediately yields that for each fixed $n\in\N$, \eqref{SCCF m n} has a unique pathwise solution $(u^{m,n},\tau^{m,n})$ such that $u^{m,n}\in C\left([0,\tau^{m,n}];H^s\right)$ $\p-a.s.$ with $s>3$.

			\subsection{Estimates on the errors}
			
			Substituting \eqref{approximation solution u-m n} into \eqref{SCCF problem}, we define the error $\EE(\omega,t,x)$ as
			\begin{align*}
				\EE(\omega,t,x)
				=\,   u_{m,n}(t,x)-u_{m,n}(0,x)+ \int_0^t(\HH u_{m,n})\partial_xu_{m,n}\, {\rm d}t'-\int_0^th(t',u_{m,n})\, {\rm d}\W\ \ \p-a.s.
			\end{align*}
			By using \eqref{approximation solution u-m n} and \eqref{low-frequency equation}, we reformulate $\EE(\omega,t,x)$ as
			\begin{align}\label{error EE}
				\EE&(\omega,t,x)\notag\\
				=\,  & u_{l}(t,x)-u_{l}(0,x)+ \int_0^t(\HH u_{l})\partial_xu_{l}\, {\rm d}t'+u_{h}(t,x)-u_{h}(0,x)\notag\\
				&+\int_0^t(\HH u_{l})\partial_xu_{h}+(\HH u_{h})(\partial_xu_{l}+\partial_xu_{h})\, {\rm d}t'-\int_0^th(t',u_{m,n})\, {\rm d}\W\notag\\
				=\,  & u_{h}(t,x)-u_{h}(0,x)
				+\int_0^t(\HH u_{l})\partial_xu_{h}+(\HH u_{h})(\partial_xu_{l}+\partial_xu_{h})\, {\rm d}t'-\int_0^th(t',u_{m,n})\, {\rm d}\W\ \ \p-a.s.
			\end{align}

			The following lemma shows the decay estimate for the low-frequency part of $u_{m,n}$.
			\begin{Lemma}\label{lemma ul Hr}
				Let $|m|=1$, $s>3$, $\delta\in(0,2)$ and $n\gg 1$. Then there exists a $T_l>0$ such that for all $n\gg 1$, the initial value problem \eqref{low-frequency equation} has a unique smooth solution $u_l=u_{l,m,n}\in C([0,T_l];H^s)$ such that $T_l$ does not depend on $n$. Besides, for all fixed $r>0$, there exists a constant $C=C_{r,\tilde{\phi},T_l}>0$ such that $u_l$ satisfies 
				\begin{equation}\label{ul Hr estimate}
					\|u_l(t)\|_{H^r}\leq C|m|n^{\frac{\delta}{2}-1},\ \  t\in[0,T_l].
				\end{equation}
				
			\end{Lemma}
			
			\begin{proof}
				For $|m|=1$ and any fixed $n\ge 1$, since $u_l(0,x)\in H^\infty$, by applying Theorem \ref{Local pathwise solution} with  $h=0$ and deterministic initial data,  we see that for any $s>3$, \eqref{low-frequency equation} has a unique (deterministic) solution $u_l=u_{l,m,n}\in C\left([0,T_l];H^s\right)$. 	
				We will show that there exists a lower bound of the existence time, i.e., there is a $T_l>0$ such that for all $n\gg 1$, $u_l=u_{l,m,n}$ exists on $[0,T_l]$ and satisfies \eqref{ul Hr estimate}. The proof of Lemma \ref{lemma ul Hr} consists of three main steps.
				
				\textit{Step 1: Estimate $\|u_l(0,x)\|_{H^r}$.} Let $g(x)=mn^{-1}\tilde{\phi}\left(\frac{x}{n^{\delta}}\right)$. For $n\gg1$, by using \eqref{H Je Hs norm}, we have that
				\begin{align*}
					\|u_l(0,x)\|^2_{H^r}=\,  \left\|-\HH\left(mn^{-1}\tilde{\phi}\left(\frac{x}{n^{\delta}}\right)\right)\right\|^2_{H^r}
					&\leq \left\|mn^{-1}\tilde{\phi}\left(\frac{x}{n^{\delta}}\right)\right\|^2_{H^r}\\
					&= \int_\R(1+|\xi|^2)^r|\widehat{g}(\xi)|^2\, {\rm d}\xi\\
					&= m^2n^{2\delta-2}
					\int_\R(1+|\xi|^2)^r\left|\widehat{\tilde{\phi}}(n^{\delta}\xi)\right|^2\, {\rm d}\xi\\
					&= m^2n^{\delta-2}\int_\R\left(1+\left|\frac{z}{n^{\delta}}\right|^2\right)^r\left|\widehat{\tilde{\phi}}(z)\right|^2\, {\rm d}z\\
					&\leq  m^2n^{\delta-2}
					\int_\R\left(1+\left|z\right|^2\right)^r\left|\widehat{\tilde{\phi}}(z)\right|^2\, {\rm d}z
					\leq C m^2n^{\delta-2},
				\end{align*}
				for some constant $C=C_{r,\tilde{\phi}}>0$. Therefore we find that
				\begin{align*}
					\|u_l(0,x)\|_{H^r}
					\leq C |m|n^{\frac{\delta}{2}-1}.
				\end{align*}
				
				\textit{Step 2: Proof  of \eqref{ul Hr estimate} for $r>3/2$.} 
				In this case, we apply Lemma \ref{Kato-Ponce commutator estimate}, \eqref{H Je Hs norm}, $H^r\hookrightarrow 
				\Wlip$ and integration by parts to find 
				\begin{align*}
					\frac{1}{2}&\frac{\, {\rm d}}{\, {\rm d}t}\|u_l\|^2_{H^r}\notag\\
					=\,  & -\int_\R D^{r}u_lD^{r}\left((\HH u_l)\partial_xu_l\right)\, {\rm d}x\\
					\leq  \, &    \left|\left(D^{r}u_l,D^{r}\left((\HH u_l)\partial_xu_l\right)\right)_{L^2}\right|\\
					\leq  \, &    \left|\left([D^{r},\HH u_l]\partial_xu_l,D^{r}u_l\right)_{L^2}\right|+\left|\left((\HH u_l)D^{r}\partial_xu_l,D^{r}u_l\right)_{L^2}\right|\\
					\lesssim \, &  \left(\|D^{r}(\HH u_l)\|_{L^2}\|\partial_xu_l\|_{L^{\infty}}
					+\|\partial_x\left(\HH u_l\right)\|_{L^{\infty}}\|D^{r-1}\partial_x u_l\|_{L^2}\right)\|u_l\|_{H^r}
					+\left|\frac12\int_\R \left(\HH u_l\right)\partial_x(D^ru_l)^2\, {\rm d}x\right|\\
					\lesssim \, &  \|\partial_xu_l\|_{L^{\infty}}\|u_l\|_{H^r}^2
					+\|\partial_x\left(\HH u_l\right)\|_{L^{\infty}}\|u_l\|_{H^r}^2\\
					\lesssim \, &  \|u_l\|_{\Wlip}\|u_l\|_{H^r}^2
					+\|\HH u_l\|_{\Wlip}\|u_l\|_{H^r}^2\\
					\leq  \, &     C\|u_l\|^{3}_{H^r},\ \ C=C_r>0.
				\end{align*}
				Solving the above inequality gives
				\begin{equation*}
					\|u_l\|_{H^r}\leq \frac{\|u_l(0)\|_{H^r}}{1-Ct\|u_l(0)\|_{H^r}},\ \ 
					0\le t<\frac{1}{C\|u_l(0)\|_{H^r}}.
				\end{equation*}
			Remember that $u_l=u_{l,m,n}$. Then we define the time interval $[0,T_{l,m,n}]$ such that
				\begin{equation}
					\|u_l\|_{H^r}\leq 2 \|u_l(0)\|_{H^r},\ \ t\in[0,T_{l,m,n}],\ \ T_{l,m,n}=\frac{1}{2C\|u_l(0)\|_{H^r}}.\label{solution bound T u_0}
				\end{equation}
				By \textit{Step 1}, we have that for $|m|=1$, 
				$T_{l,m,n}\gtrsim \frac{1}{2Cn^{\frac{\delta}{2}-1}}\rightarrow\infty,\ \text{as}\ n\rightarrow\infty.
				$
				Therefore we can find a common time interval $[0,T_{l}]$ such that 
				\begin{equation}\label{u-l r>3/2}
					\|u_l\|_{H^r}\leq 2 \|u_l(0)\|_{H^r}\leq C|m|n^{\frac{\delta}{2}-1},\ \ t\in[0,T_l],\ \ C=C_{r,\tilde{\phi}}>0,
				\end{equation}
				which is \eqref{ul Hr estimate}.
				
				\textit{Step 3: Proof of \eqref{ul Hr estimate} for $0<r\leq 3/2$.} Applying Lemma \ref{Kato-Ponce commutator estimate}, \eqref{H Je Hs norm}, we have
				\begin{align*}
					&\frac{1}{2}\frac{\, {\rm d}}{\, {\rm d}t}\|u_l\|^2_{H^r}\\
					=\,  & -\int_\R D^{r}u_lD^{r}\left((\HH u_l)\partial_xu_l\right)\, {\rm d}x\\
					\leq  \, &    \left|\left(D^{r}u_l,D^{r}\left((\HH u_l)\partial_xu_l\right)\right)_{L^2}\right|\\
					\leq  \, &    \left|\left([D^{r},\HH u_l]\partial_xu_l,D^{r}u_l\right)_{L^2}\right|+\left|\left((\HH u_l)D^{r}\partial_xu_l,D^{r}u_l\right)_{L^2}\right|\\
					\lesssim \, &  \left(\|D^{r}(\HH u_l)\|_{L^2}\|\partial_xu_l\|_{L^{\infty}}
					+\|\partial_x\left(\HH u_l\right)\|_{L^{\infty}}\|D^{r-1}\partial_x u_l\|_{L^2}\right)\|u_l\|_{H^r}
					+\left|\frac12\int_\R \left(\HH u_l\right)\partial_x(D^ru_l)^2\, {\rm d}x\right|\\
					\lesssim \, &  \|\partial_xu_l\|_{L^{\infty}}\|u_l\|_{H^r}^2
					+\|\partial_x\left(\HH u_l\right)\|_{L^{\infty}}\|u_l\|_{H^r}^2.
				\end{align*}
				It follows from the embedding $H^{r+\frac32}\hookrightarrow W^{1,\infty}$ that
				\begin{align*}
					\frac{1}{2}\frac{\, {\rm d}}{\, {\rm d}t}\|u_l\|^2_{H^r}	\lesssim \, &  \|\partial_xu_l\|_{L^{\infty}}\|u_l\|_{H^r}^2
					+\|\partial_x\left(\HH u_l\right)\|_{L^{\infty}}\|u_l\|_{H^r}^2\\
					\lesssim \, &  \|u_l\|_{\Wlip}\|u_l\|^2_{H^r}
					+\|\HH u_l\|_{\Wlip}\left\|u_l\right\|_{H^r}^2\\
					\lesssim \, &  	\|u_l\|_{H^{r+\frac{3}{2}}}\|u_l\|_{H^{r}}^2+\|\HH u_l\|_{H^{r+\frac{3}{2}}}\|u_l\|_{H^{r}}^2\\
					\lesssim \, &  	\|u_l\|_{H^{r+\frac{3}{2}}}\|u_l\|_{H^{r}}^2.
				\end{align*}
				Using the conclusion of \textit{Step 2} for $r+\frac{3}{2}>\frac{3}{2}$, we have
				\begin{align*}
					\frac{\, {\rm d}}{\, {\rm d}t}\|u_l\|_{H^r}\lesssim\|u_l\|_{H^{r}}\|u_l(0)\|_{H^{r+\frac{3}{2}}},\ \ t\in[0,T_l],
				\end{align*}
				and hence
				\begin{align*}
					\|u_l(t)\|_{H^r}\lesssim \|u_l(0)\|_{H^r}+\int_0^t\|u_l\|_{H^{r}}\|u_l(0)\|_{H^{r+\frac{3}{2}}}\, {\rm d}t,\ \ t\in[0,T_l].
				\end{align*}
				Applying Gr\"{o}nwall's inequality to the above inequality, we have
				\begin{align*}
					\|u_l\|_{H^r}\lesssim\|u_l(0)\|_{H^r}\exp\left\{\|u_l(0)\|_{H^{r+\frac{3}{2}}}T_l\right\},\ \ t\in[0,T_l].
				\end{align*}
				Since $\delta\in(0,2)$, we can infer from \textit{Step 1} that $\exp\left\{\|u_l(0)\|_{H^{r+\frac{3}{2}}}T_l\right\} <C(r,\tilde{\phi},T_l)$ for some constant $C(r,\tilde{\phi},T_l)>0$. 
				Therefore we see that there exists a constant $C=C_{r,\tilde{\phi},T_l}>0$ such that
				\begin{align*}
					\|u_l\|_{H^r}\leq C|m|n^{\frac{\delta}{2}-1},\ \ t\in[0,T_l],
				\end{align*}
				concluding the desired estimate \eqref{ul Hr estimate}. 
			\end{proof}
			
			The above result implies that the $H^s$-norm of $u_l$, the low-frequency part of the approximate solution defined by \eqref{approximation solution u-m n}, is decaying. For the high-frequency part $u_h$, due to Lemma \ref{lemma uh Hr}, its $H^s$-norm  is bounded. To sum up, let $T_l$ be given in Lemma \ref{lemma ul Hr}, for any $s>0$, there is a constant  $M=M_{s,\tilde{\phi},\phi,T_l}>0$ such that
			\begin{equation}\label{u-m n Hs}
				\|u_{m,n}(t)\|_{H^s}\lesssim M,\ \ t\in[0,T_l].
			\end{equation}
		Moreover, although not strictly necessary, we can infer from \eqref{u-l r>3/2} that $M$ can be independent of $T_l$ when $s>3/2$.
			
			\subsubsection{Estimating the error $\EE$}
			Recall \eqref{error EE}. By using \eqref{phi phi-tilde}, we have that $\phi=\tilde{\phi}\phi$. Then by \eqref{high-frequency approximation solutions} and $u_l(0,x)$ in \eqref{low-frequency equation}, we see that for $m\in\{-1,1\}$, 
			\begin{align*}
				u_{h}(t,x)-&u_{h}(0,x)\notag\\
				=\,  & n^{-\frac{\delta}{2}-s}\phi\left(\frac{x}{n^{\delta}}\right)\cos(nx-mt)-n^{-\frac{\delta}{2}-s}\phi\left(\frac{x}{n^{\delta}}\right)\cos(nx)\\
				=\,  & m^{-1}m\tilde{\phi}\left(\frac{x}{n^{\delta}}\right)n^{-\frac{\delta}{2}-s}\phi\left(\frac{x}{n^{\delta}}\right)\cos(nx-mt)-m^{-1}m
				\tilde{\phi}\left(\frac{x}{n^{\delta}}\right)
				n^{-\frac{\delta}{2}-s}\phi\left(\frac{x}{n^{\delta}}\right)\cos(nx)\notag\\
				=\,  & m^{-1}\left(\HH u_l(0,x)\right)n^{1-\frac{\delta}{2}-s}\phi\left(\frac{x}{n^{\delta}}\right)\cos(nx-mt)-m^{-1}\left(\HH u_l(0,x)\right)n^{1-\frac{\delta}{2}-s}\phi\left(\frac{x}{n^{\delta}}\right)\cos(nx)\\
				=\,  & \int_0^t\left(\HH u_l(0,x)\right)n^{1-\frac{\delta}{2}-s}\phi\left(\frac{x}{n^{\delta}}\right)\sin(nx-mt')\, {\rm d}t'.
			\end{align*}
			Furthermore,
			\begin{align}
				\int_0^t(\HH u_{l})\partial_xu_h\, {\rm d}t'=\, & -\int_0^t(\HH u_{l})(t')n^{1-\frac{\delta}{2}-s}\phi\left(\frac{x}{n^{\delta}}\right)\sin(nx-mt')\, {\rm d}t'\notag\\
				&+\int_0^t(\HH u_{l})(t')n^{-\frac{3\delta}{2}-s}\partial_x\phi\left(\frac{x}{n^{\delta}}\right)\cos(nx-mt')\, {\rm d}t'\label{ul uh_x}.
			\end{align}	
			Thus, \eqref{error EE} becomes
			\begin{align}
				\EE(\omega,t,x)
				=\,  & \int_0^t[(\HH u_{l})(0)-(\HH u_{l})(t')]n^{1-\frac{\delta}{2}-s}\phi\left(\frac{x}{n^{\delta}}\right)\sin(nx-mt')\, {\rm d}t'\notag\\
				&+\int_0^t(\HH u_{l})(t')n^{-\frac{3\delta}{2}-s}\partial_x\phi\left(\frac{x}{n^{\delta}}\right)\cos(nx-mt')\, {\rm d}t'\notag\\
				&+\int_0^t(\HH u_{h})(\partial_xu_l+\partial_xu_h)\, {\rm d}t'-\int_0^th(t',u_{m,n})\, {\rm d}\W\notag\\
				=\,  & \int_0^tE \, {\rm d}t'-\int_0^th(t',u_{m,n})\, {\rm d}\W\ \ \p-a.s.,\label{Error}
			\end{align}	
			where
			\begin{align}
				E=\, E(\omega,t,x)=\,& [(\HH u_{l})(0)-(\HH u_{l})(t)]n^{1-\frac{\delta}{2}-s}\phi\left(\frac{x}{n^{\delta}}\right)\sin(nx-mt)\notag\\
				&+(\HH u_{l})(t)n^{-\frac{3\delta}{2}-s}\partial_x\phi\left(\frac{x}{n^{\delta}}\right)\cos(nx-mt)
				+(\HH u_{h})(\partial_xu_l+\partial_xu_h).\label{Error E}
			\end{align}

			Now we shall estimate the $H^{\sigma_0}$-norm of the error $\EE$, where $\sigma_0$ is given in Assumption \ref{Assumption-2}. Actually, we will show that the $H^{\sigma_0}$-norm of $\EE$ is decaying.
			
			\begin{Lemma}\label{Lemma Error estimate}
				Let $n\gg 1$, $s>3$, $\frac34<\delta<1$. Let $T_l$ be given in Lemma \ref{lemma ul Hr}, and $\sigma_0$ be given in Assumption \ref{Assumption-2}. Let
				\begin{equation}\label{rs}
				r_s=-s-1+\sigma_0+\delta.
				\end{equation}
				Then $r_s<0$ and the error $\EE$ given by \eqref{Error} satisfies
				\begin{align*}
					\E\sup_{t\in[0,T_l]}\|\EE(t)\|_{H^{\sigma_0}}^2\leq Cn^{2r_s},
				\end{align*}	
			where $C=C(\sigma_0,\tilde{\phi},\phi,T_l)>0$ is a constant independent of $n$. 
			\end{Lemma}
			
			\begin{proof}
				It is obvious by construction that $r_s$ given by \eqref{rs} is negative. 
				Combining \eqref{Error E}, the embedding $H^{\sigma_0}\hookrightarrow L^\infty$ and Lemmas \ref{lemma ul Hr} and \ref{lemma uh Hr}, we find that for $t\in[0,T_l]$,
				\begin{align}
					\|E\|_{H^{\sigma_0}}
					\leq \,	&
					\left\|\left[(\HH u_{l})(0)-(\HH u_{l})(t)\right]n^{1-\frac{\delta}{2}-s}\phi\left(\frac{x}{n^{\delta}}\right)\sin(nx-mt)\right\|_{H^{\sigma_0}}\notag\\
					&+\left\|(\HH u_{l})(t)n^{-\frac{3\delta}{2}-s}\partial_x\phi\left(\frac{x}{n^{\delta}}\right)\cos(nx-mt)\right\|_{H^{\sigma_0}}\notag\\
					&+\left\|(\HH u_{h})\partial_xu_l\right\|_{H^{\sigma_0}}+\left\|(\HH u_{h})\partial_xu_h\right\|_{H^{\sigma_0}}\notag\\
					\lesssim \, &   n^{1-\frac{\delta}{2}-s}\left\|(\HH u_{l})(0)-(\HH u_{l})(t)\right\|_{H^{\sigma_0}}\left\|\phi\left(\frac{x}{n^{\delta}}\right)\sin(nx-mt)\right\|_{H^{\sigma_0}}\notag\\
					&+n^{-\frac{3\delta}{2}-s}\left\|(\HH u_{l})(t)\right\|_{H^{\sigma_0}}
					\left\|\partial_x\phi\left(\frac{x}{n^{\delta}}\right)\cos(nx-mt')\right\|_{H^{\sigma_0}}\notag\\		
					&+\left\|(\HH u_{h})\partial_xu_l\right\|_{H^{\sigma_0}}+\left\|(\HH u_{h})\partial_xu_h\right\|_{H^{\sigma_0}}\notag\\
					\lesssim \, &  n^{1-s+\sigma_0}\left\|(\HH u_{l})(0)-(\HH u_{l})(t)\right\|_{H^{\sigma_0}}
					+n^{-s-1+\sigma_0-\frac{1}{2}\delta}
					+\left\|(\HH u_{h})\partial_xu_l\right\|_{H^{\sigma_0}}+\left\|(\HH u_{h})\partial_xu_h\right\|_{H^{\sigma_0}}\notag\\
					\lesssim \, &  n^{1-s+\sigma_0}\left\|(\HH u_{l})(0)-(\HH u_{l})(t)\right\|_{H^{\sigma_0}}
					+n^{r_s}
					+\left\|(\HH u_{h})\partial_xu_l\right\|_{H^{\sigma_0}}+\left\|(\HH u_{h})\partial_xu_h\right\|_{H^{\sigma_0}}.\label{Error E estimate}
				\end{align}
				For $\left\|(\HH u_{l})(0)-(\HH u_{l})(t)\right\|_{H^{\sigma_0}}$, it follows from the fundamental theorem of calculus and $H^{\sigma_0}\hookrightarrow L^\infty$  that for $t\in[0,T_l]$,
				\begin{align*}
					\left\|(\HH u_{l})(0)-(\HH u_{l})(t)\right\|_{H^{\sigma_0}}
					\lesssim \, &  \|u_l(0)-u_l(t)\|_{H^{\sigma_0}}\\
					=\, &\left\|\int_0^t\partial_tu_l(t')\, {\rm d}t'\right\|_{H^{\sigma_0}}\\
					\lesssim \,& \int_0^t\|(\HH u_{l})\partial_xu_l\|_{H^{\sigma_0}}\, {\rm d}t'\\
					\lesssim \, &  \int_0^t	\|u_l\|_{H^{\sigma_0+1}}^{2}\, {\rm d}t'
					\lesssim \, n^{\delta-2}T_l,
				\end{align*}
				where we used \eqref{low-frequency equation} with $t\in[0,T_l]$, Lemma \ref{lemma ul Hr} and the embedding $H^{\sigma_0+1}\hookrightarrow \Wlip$. Therefore, 
				\begin{align}
					n^{1-s+\sigma_0}\left\|(\HH u_{l})(0)-(\HH u_{l})(t)\right\|_{H^{\sigma_0}}\lesssim n^{1-s+\sigma_0+\delta-2}T_l=n^{r_s}T_l,\ \ t\in[0,T_l].\label{E part 1}
				\end{align}
				Next, applying Lemma \ref{lemma ul Hr} and \ref{lemma uh Hr}, we have for $t\in[0,T_l]$,
				\begin{align}
					\left\|(\HH u_{h})\partial_xu_l\right\|_{H^{\sigma_0}}
					\lesssim \, &  \|u_{h}\|_{H^{\sigma_0}}\|u_l\|_{H^{\sigma_0+1}}\notag\\
					\lesssim \, &  n^{-s+\sigma_0}n^{\frac\delta2-1}
					=\, n^{-s+\sigma_0+\frac\delta2-1}	\lesssim n^{r_s},\label{E part 3}\\
					\left\|(\HH u_{h})\partial_xu_h\right\|_{H^{\sigma_0}}
					\lesssim \, &    \|u_{h}\|_{H^{\sigma_0}}\|u_h\|_{H^{\sigma_0+1}}\notag\\
					\lesssim \, &  n^{-s+\sigma_0}n^{-s+\sigma_0+1}
					=\, n^{-2s+2\sigma_0+1}	\lesssim n^{r_s}. \label{E part 4}
				\end{align}
				Here in \eqref{E part 4} we used the assumption $\sigma_0\in(3/2,7/4)$ to guarantee $n^{-2s+2\sigma_0+1}	\lesssim n^{r_s}$. Inserting \eqref{E part 1}, \eqref{E part 3} and \eqref{E part 4} into \eqref{Error E estimate}, we finally obtain
				\begin{align}\label{error E estimate}
					\|E\|_{H^{\sigma_0}}\lesssim n^{r_s}, \ \ t\in[0,T_l].
				\end{align}
With \eqref{error E estimate} at hand, we are in the position to estimate $	\E\sup_{t\in[0,T_l]}\|\EE(t)\|_{H^{\sigma_0}}^2$.  Invoking It\^{o} formula in \eqref{Error} leads to
				\begin{align*}
					\|\EE(t,x)\|_{H^{\sigma_0}}^2\leq \  & \left|-2\int_0^t(h(t',u_{m,n})\, {\rm d}\W,\EE)_{H^{\sigma_0}}\right|+2\int_0^t\left|(E,\EE)_{H^{\sigma_0}}\right|\, {\rm d}t'
					+\int_0^t\|h(t',u_{m,n})\|_{\LL_2(\U;H^{\sigma_0})}^2\, {\rm d}t'.
				\end{align*}
				Taking supremum with respect to $t\in[0,T_l]$ and using the BDG inequality, we can find some $\overline{C}>0$ such that
				\begin{align*}
					&\E\sup_{t\in[0,T_l]}\|\EE(t)\|_{H^{\sigma_0}}^2\\
					\leq \,&\frac 12\E\sup_{t\in[0,T_l]}\|\EE(t)\|_{H^{\sigma_0}}^2
					+\overline{C}\E\int_0^{T_l}\|h(t,u_{m,n})\|_{\LL_2(\U;H^{\sigma_0})}^2\, {\rm d}t
					+\overline{C}\int_0^{T_l}\left[\E\|E\|_{H^{\sigma_0}}^2+\E\|\EE(t)\|_{H^{\sigma_0}}^2\right]\, {\rm d}t.
				\end{align*}
				By virtue of \eqref{error E estimate}, we arrive at
				\begin{align*}
					\E\sup_{t\in[0,T_l]}\|\EE(t)\|_{H^{\sigma_0}}^2
					\lesssim T_l n^{2r_s}
					+\E\int_0^{T_l}\|h(t,u_{m,n})\|_{\LL_2(\U;H^{\sigma_0})}^2\, {\rm d}t
					+\int_0^{T_l}\E\sup_{t'\in[0,t]}\|\EE(t')\|_{H^{\sigma_0}}^2\, {\rm d}t.
				\end{align*}
				Now, we estimate $\|h(t,u_{m,n})\|_{\LL_2(\U;H^{\sigma_0})}$. For any fixed $s>3$, on account of Assumption \ref{Assumption-2}, Lemmas \ref{lemma uh Hr} and \ref{lemma ul Hr}, we can pick $\kappa> 2(1+\frac{s-\sigma_0-1}{2-\delta})$ such that
				\begin{align*}
					\|h(t,u_{m,n})\|_{\LL_2(\U;H^{\sigma_0})}^2
					\lesssim \left({\rm e}^\frac{-1}{\|u_{m,n}\|_{H^{\sigma_0}}}\right)^2 
					\lesssim  \left(n^{-s+\sigma_0}+n^{\frac{\delta}{2}-1}\right)^{2\kappa}
					\lesssim  n^{2r_s},
				\end{align*}
which gives
				\begin{align*}
					\E\sup_{t\in[0,T_l]}\|\EE(t)\|_{H^{\sigma_0}}^2
					\lesssim T_l n^{2r_s}+\int_0^{T_l}\E\sup_{t'\in[0,t]}\|\EE(t')\|_{H^{\sigma_0}}^2\, {\rm d}t.
				\end{align*}
				Obviously, for each $n\geq 1$, $\E\sup_{t\in[0,T_l]}\|\EE(t)\|_{H^{\sigma_0}}^2$ is finite and $T_l>0$ is fixed. Then by the Gr\"{o}nwall inequality, we have
				\begin{align*}
					\E\sup_{t\in[0,T_l]}\|\EE(t)\|_{H^{\sigma_0}}^2\leq Cn^{2r_s}.
				\end{align*}
				The proof is completed.
			\end{proof}
			
			\subsubsection{Estimating $u_{m,n}-u^{m,n}$}
			
			Recall the approximate solutions $u_{m,n}$ given by \eqref{approximation solution u-m n}. Then we have the following estimates on the difference between the actual solutions and the approximate solutions.
			\begin{Lemma}\label{difference estimate lemma}
				Let $s>3$, $\frac34<\delta<1$, $\sigma_0$ be given in Assumption \ref{Assumption-2} and $r_s<0$ be given in \eqref{rs}. For any $R>1$,  we define
				\begin{align}\label{tau actual solution mnR}
					\tau^{m,n}_R:=\inf\{t>0:\|u^{m,n}\|_{H^{s}}>R\}.
				\end{align}
				Then for $n\gg 1$ and $T_l>0$ given in Lemma \ref{lemma ul Hr},  
				\begin{align}
					&\E\sup_{t\in[0,T_l\wedge\tau^{m,n}_R]}\|u_{m,n}-u^{m,n}\|_{H^{\sigma_0}}^2 \leq Cn^{2r_s},\label{Error sigma}
				\end{align}
				\begin{align}
					&\E\sup_{t\in[0,T_l\wedge\tau^{m,n}_R]}\|u_{m,n}-u^{m,n}\|_{H^{2s-\sigma_0}}^2 \leq Cn^{2s-2\sigma_0},\label{Error 2s minues sigma}
				\end{align}
				where $C=C(s,\sigma_0,\tilde{\phi},\phi,T_l,R)>0$ is a constant independent of $n$.	
			\end{Lemma}
			
			\begin{proof}
				Let $v=v_{m,n}=u_{m,n}-u^{m,n}$. Then $v$ satisfies $v(0)=0$ and
\begin{align*}
	v(t)+\int_0^t\left(\left(\HH v\right)\partial_xu_{m,n}+\left(\HH u^{m,n}\right)\partial_xv\right)\, {\rm d}t'=-\int_0^th(t',u^{m,n})\, {\rm d}\W+\int_0^tE\, {\rm d}t'.
\end{align*}
				where \eqref{Error} is used. 
	For $T_l>0$, we use the It\^{o} formula to find
				\begin{align*}
					\|v(t)\|_{H^{\sigma_0}}^2=\, & -2\int_0^t\left(h(t',u^{m,n})\, {\rm d}\W,v\right)_{H^{\sigma_0}}+2\int_0^t(E,v)_{H^{\sigma_0}}\, {\rm d}t'-2\int_0^t\left(\left(\HH v\right)\partial_xu_{m,n},v\right)_{H^{\sigma_0}}\, {\rm d}t'\\
					&-2\int_0^t\left(\left(\HH u^{m,n}\right)\partial_xv,v\right)_{H^{\sigma_0}}\, {\rm d}t'+\int_0^t\|h(t',u^{m,n})\|_{\LL_2(\U;H^{\sigma_0})}^2\, {\rm d}t'.
				\end{align*}
				Taking supremum with respect to $t\in[0,T_l\wedge\tau^{m,n}_R]$, and then using the BDG inequality yield that for some $\overline{C}>0$,
				\begin{align*}
					&\E\sup_{t\in[0,T_l\wedge\tau^{m,n}_R]}\|v(t)\|_{H^{\sigma_0}}^2\\
					\lesssim  \, &    C\E\left(\int_0^{T_l\wedge\tau^{m,n}_R}\|v\|_{H^{\sigma_0}}^2\|h(t,u^{m,n})\|_{\LL_2(\U;H^{\sigma_0})}^2\, {\rm d}t\right)^{1/2}+2\E\int_0^{T_l\wedge\tau^{m,n}_R}\left|(E,v)_{H^{\sigma_0}}\right|\, {\rm d}t\\
					&+2\E\int_0^{T_l\wedge\tau^{m,n}_R}\left|\left(\left(\HH v\right)\partial_xu_{m,n},v\right)_{H^{\sigma_0}}\right|\, {\rm d}t+2\E\int_0^{T_l\wedge\tau^{m,n}_R}\left|\left(\left(\HH u^{m,n}\right)\partial_xv,v\right)_{H^{\sigma_0}}\right|\, {\rm d}t\\
					&+\E\int_0^{T_l\wedge\tau^{m,n}_R}\|h(t,u^{m,n})\|_{\LL_2(\U;H^{\sigma_0})}^2\, {\rm d}t\\
					\leq  \, &    \frac12\E\sup_{t\in[0,T_l\wedge\tau^{m,n}_R]}\|v(t)\|_{H^{\sigma_0}}^2
					+\overline{C}\E\int_0^{T_l\wedge\tau^{m,n}_R}\left|(E,v)_{H^{\sigma_0}}\right|\, {\rm d}t\\
					&+\overline{C}\E\int_0^{T_l\wedge\tau^{m,n}_R}\left|\left(\left(\HH v\right)\partial_xu_{m,n},v\right)_{H^{\sigma_0}}\right|\, {\rm d}t+\overline{C}\E\int_0^{T_l\wedge\tau^{m,n}_R}\left|\left(\left(\HH u^{m,n}\right)\partial_xv,v\right)_{H^{\sigma_0}}\right|\, {\rm d}t\\
			&+\overline{C}\E\int_0^{T_l\wedge\tau^{m,n}_R}\|h(t,u^{m,n})\|_{\LL_2(\U;H^{\sigma_0})}^2\, {\rm d}t.
				\end{align*}
Recall that \eqref{u-m n Hs} gives $\|u_{m,n}\|_{H^{s}}\lesssim M_{s,\tilde{\phi},\phi,T_l}$ on $t\in[0,T_l\wedge\tau^{m,n}_R]$. Hence we can infer from Assumption \ref{Assumption-2} that for some $\overline{C}>0$
				\begin{align*}
					\|h(t,u^{m,n})\|_{\LL_2(\U;H^{\sigma_0})}^2
					\lesssim \, &  \|h(t,u_{m,n})\|_{\LL_2(\U;H^{\sigma_0})}^2+\|h(t,u_{m,n})-h(t,u^{m,n})\|_{\LL_2(\U;H^{\sigma_0})}^2\\
					\leq \, & \overline{C}\left({\rm e}^\frac{-1}{\|u_{m,n}\|_{H^{\sigma_0}}}\right)^2
					+q(\overline{C}) \|v\|_{H^{\sigma_0}}^2,\ \ t\in[0,T_l\wedge\tau^{m,n}_R]\ \ \p-a.s.,
				\end{align*}
				where $q(\cdot)$ is given in \eqref{assumption 2 for h}. As a result, for any fixed $s>3$, by applying Lemmas \ref{lemma uh Hr} and \ref{lemma ul Hr} again, we can pick $\lambda>2(1+\frac{s-\sigma_0-1}{2-\delta})$ to derive
				\begin{align*}
					\|h(t,u^{m,n})\|_{\LL_2(\U;H^{\sigma_0})}^2
					\lesssim \, &  \|u_{m,n}\|^{2\lambda}_{H^{\sigma_0}}
					+ \|v\|_{H^{\sigma_0}}^2\\
					\lesssim \, &  \left(n^{-s+\sigma_0}+n^{\frac{\delta}{2}-1}\right)^{2\lambda}
					+  \|v\|_{H^{\sigma_0}}^2 
					\lesssim\,  n^{2r_s}+  \|v\|_{H^{\sigma_0}}^2,	\ \ t\in[0,T_l\wedge\tau^{m,n}_R]\ \ \p-a.s.
				\end{align*}
				Via \eqref{error E estimate}, we have
				\begin{align*}
					2\left|(E,v)_{H^{\sigma_0}}\right|\leq 2\|E\|_{H^{\sigma_0}}\|v\|_{H^{\sigma_0}}\lesssim\|E\|_{H^{\sigma_0}}^2+\|v\|_{H^{\sigma_0}}^2\lesssim  n^{2r_s}+\|v\|_{H^{\sigma_0}}^2.
				\end{align*}
				Using Lemma \ref{Kato-Ponce commutator estimate},  \eqref{u-m n Hs}, \eqref{H Je Hs norm}, integration by parts, and the embedding $H^{s}\hookrightarrow H^{\sigma_0}\hookrightarrow\Wlip$, we obtain that the almost surely
				\begin{align*}
					\left|\left(\left(\HH v\right)\partial_xu_{m,n},v\right)_{H^{\sigma_0}}\right|
					\lesssim\, \|\HH v\|_{H^{\sigma_0}}\|u_{m,n}\|_{H^{\sigma_0+1}}\|v\|_{H^{\sigma_0}}
					\lesssim\,  \|u_{m,n}\|_{H^{s}}\|v\|_{H^{\sigma_0}}^2 \lesssim\|v\|_{H^{\sigma_0}}^2,	\ t\in[0,T_l\wedge\tau^{m,n}_R],
				\end{align*}
				and 
				\begin{align*}
					&\left|\left(\left(\HH u^{m,n}\right)\partial_xv,v\right)_{H^{\sigma_0}}\right|\\
					=\,  & \left|([D^{\sigma_0},\HH u^{m,n}]\partial_xv,D^{\sigma_0}v)_{L^2}
					+(\HH u^{m,n}D^{\sigma_0}\partial_xv,D^{\sigma_0}v)_{L^2}\right|\\
					\lesssim \, &  \left(\|D^{\sigma_0}\HH u^{m,n}\|_{L^2}\|\partial_xv\|_{L^\infty}+\|\partial_x\HH u^{m,n}\|_{L^\infty}\|D^{\sigma_0-1}\partial_xv\|_{L^2}\right)\|v\|_{H^{\sigma_0}}+\|\partial_{x}\HH u^{m,n}\|_{L^\infty}\|v\|_{H^{\sigma_0}}^2\\
					\lesssim \, &  \|\HH u^{m,n}\|_{H^s}\|v\|_{H^{\sigma_0}}^2
					\lesssim\|u^{m,n}\|_{H^s}\|v\|_{H^{\sigma_0}}^2\lesssim\|v\|_{H^{\sigma_0}}^2,	\ t\in[0,T_l\wedge\tau^{m,n}_R].
				\end{align*}
				To sum up, we obtain that 
				\begin{align*}
					\E\sup_{t\in[0,T_l\wedge\tau^{m,n}_R]}\|v(t)\|_{H^{\sigma_0}}^2
					\lesssim T_ln^{2r_s}+\int_0^{T_l}\E\sup_{t'\in[0,t\wedge\tau^{m,n}_R]}\|v(t')\|_{H^{\sigma_0}}^2\, {\rm d}t.
				\end{align*}
				Using the Gr\"{o}nwall inequality, we obtain \eqref{Error sigma}. 
				
				Now we prove \eqref{Error 2s minues sigma}. To this end, we first notice that $2s-\sigma_0>3$ and $u_{m,n}$ is the unique solution to \eqref{SCCF m n}. Then, similar to \eqref{u Hs estimate}, we can use \eqref{tau actual solution mnR} and Assumption \ref{Assumption-2} to find for each fixed $n\in\N$ that
				\begin{align*}
					\E\sup_{t\in[0,T_l\wedge\tau^{m,n}_R]}\|u^{m,n}\|_{H^{2s-\sigma_0}}^2
					\lesssim \E\|u_{m,n}(0)\|_{H^{2s-\sigma_0}}^2
					+ \int_0^{T_l}\E\sup_{t'\in[0,t\wedge\tau^{m,n}_R]}\|u^{m,n}\|_{H^{2s-\sigma_0}}^2\, {\rm d}t.
				\end{align*}
				Using the Gr\"{o}nwall inequality and Lemmas \ref{lemma ul Hr} and \ref{lemma uh Hr}, we have
				\begin{align*}
					&\E\sup_{t\in[0,T_l\wedge\tau^{m,n}_R]}\|u^{m,n}\|_{H^{2s-\sigma_0}}^2
					\lesssim \E\|u_{m,n}(0)\|_{H^{2s-\sigma_0}}^2
					\lesssim (n^{\frac\delta2-1}+n^{s-\sigma_0})^2\lesssim n^{2s-2\sigma_0},\ \ n\geq1.
				\end{align*}
				Hence, by Lemmas \ref{lemma ul Hr} and \ref{lemma uh Hr} again, we arrive at 
				\begin{align*}
					&\E\sup_{t\in[0,T_l\wedge\tau^{m,n}_R]}\|u^{m,n}-u_{m,n}\|_{H^{2s-\sigma_0}}^2\\
					\leq\ & 2\E\sup_{t\in[0,T_l\wedge\tau^{m,n}_R]}\|u^{m,n}\|_{H^{2s-\sigma_0}}^2
					+2\E\sup_{t\in[0,T_l\wedge\tau^{m,n}_R]}\|u_{m,n}\|_{H^{2s-\sigma_0}}^2 
					\leq\  Cn^{2s-2\sigma_0},\ \ n\geq1.
				\end{align*}
				Therefore, we complete the proof.
			\end{proof}

			\subsection{Concluding the proof of Theorem \ref{Weak instability}}\label{conclude:weak:theorem}
			To begin with, we have the following property:
			\begin{Lemma}\label{exiting time infty lemma}
				Let Assumption \ref{Assumption-2} hold true. Suppose that for some $R_0\gg 1$, the $R_0$-exiting time of the zero solution to \eqref{SCCF problem} is strongly stable. Then we have
				\begin{equation}\label{tau actual solution mnR infty}
					\lim_{n\rightarrow\infty}\tau^{m,n}_{R_0}=\infty\ \p-a.s.
				\end{equation}	
			\end{Lemma}
			\begin{proof}
				By Assumption \ref{Assumption-2}, the unique solution with zero initial data to \eqref{SCCF problem} is zero.
				Now we notice that for all $s'<s$, $\lim_{n\rightarrow\infty}\|u_{m,n}(0)\|_{H^{s'}}=\lim_{n\rightarrow\infty}\|u_{m,n}(0)-0\|_{H^{s'}}=0$ and the $R_0$-exiting time of the zero solution is $\infty$. Then the assumption that $R_0$-exiting time of the zero solution to \eqref{SCCF problem} is strongly stables immediately implies \eqref{tau actual solution mnR infty}.
			\end{proof}
			
			\begin{proof}[Proof of Theorem \ref{Weak instability}]
				We only need to show that  if  the $R_0$-exiting time is strongly stable at the zero solution  for some $R_0\gg1$, then $\{u^{-1,n}\}$ and $\{u^{1,n}\}$ are two sequences of pathwise solutions such that \eqref{tau 1 2 n}, \eqref{sup u}, \eqref{same initail data} and \eqref{sup sin t} are satisfied.
				
				For each $n>1$ and for fixed $R_0\gg 1$, Lemmas \ref{lemma ul Hr}, \ref{lemma uh Hr} and \eqref{tau actual solution mnR} give $\p\{\tau^{m,n}_{R_0}>0\}=1$, and  Lemma \ref{exiting time infty lemma} implies \eqref{tau 1 2 n}. Then, it follows from		
				Theorem \ref{Local pathwise solution} and \eqref{tau actual solution mnR} that $u^{m,n}\in C([0,\tau^{m,n}_{R_0}];H^s)$ $\p-a.s.$ and \eqref{sup u} holds true. 
				Next, we check \eqref{same initail data}.  By interpolation, we have
				\begin{align*}
					\, &\E\sup_{t\in[0,T_l\wedge\tau^{m,n}_{R_0}]}\|u_{m, n}-u^{m, n}\|_{H^{s}}\\
					\lesssim\, &
					\left( \E\sup_{t\in[0,T_l\wedge\tau^{m,n}_{R_0}]}\|u_{m, n}-u^{m, n}\|_{H^{\sigma_0}}\right)^{\frac12}
					\left( \E\sup_{t\in[0,T_l\wedge\tau^{m,n}_{R_0}]}\|u_{m, n}-u^{m, n}\|_{H^{2s-\sigma_0}}\right)^{\frac12}\\
					\lesssim\, &
					\left( \E\sup_{t\in[0,T_l\wedge\tau^{m,n}_{R_0}]}\|u_{m, n}-u^{m, n}\|^2_{H^{\sigma_0}}\right)^{\frac14}
					\left( \E\sup_{t\in[0,T_l\wedge\tau^{m,n}_{R_0}]}\|u_{m, n}-u^{m, n}\|^2_{H^{2s-\sigma_0}}\right)^{\frac14}.
				\end{align*} 
			Combining Lemma \ref{difference estimate lemma} and the above estimate yields
				\begin{align}
					\E\sup_{t\in[0,T_l\wedge\tau^{m,n}_{R_0}]}
					\|u_{m, n}-u^{m, n}\|_{H^{s}}
						\lesssim\, & n^{\frac{1}{4}\cdot 2r_{s}}\cdot n^{\frac{1}{4}\cdot(2s-2\sigma_0)}=n^{r'_s},\label{u difference r's}
				\end{align}
				where $r_s$ is defined by \eqref{rs} and
				\begin{equation*}
					0>r'_s=r_s\cdot \frac{1}{2}+(s-\sigma_0)\cdot\frac{1}{2}=\frac{\delta-1}{2}.
				\end{equation*}
				Since $r'_s<0$, we can deduce that
				\begin{align}
					\lim_{n\rightarrow\infty}
					\E\sup_{t\in[0,T_l\wedge\tau^{m,n}_{R_0}]}\|u_{m, n}-u^{m, n}\|_{H^{s}}
					=0.\label{u difference tends to 0}
				\end{align}
	Since $\delta<1$, we have
				\begin{align*} 
					\|u^{-1,n}(0)-u^{1,n}(0)\|_{H^{s}}^2
					=\, &\|u_{-1,n}(0)-u_{1,n}(0)\|_{H^{s}}^2\\
					=\,& \left\|2\HH\left(n^{-1}\tilde{\phi}\left(\frac{x}{n^{\delta}}\right)\right)\right\|_{H^{s}}^2
					\lesssim \left\|n^{-1}\tilde{\phi}\left(\frac{x}{n^{\delta}}\right)\right\|_{H^{s}}^2
					\lesssim n^{\frac{\delta}{2}-1}\|\tilde{\phi}\|_{H^s}\rightarrow 0, {~\rm as~}n\rightarrow \infty,
				\end{align*}
				which implies that \eqref{same initail data} holds true.
				
				Now we prove \eqref{sup sin t}. Let $T_l>0$ be given in Lemma \ref{lemma ul Hr}. We can infer from \eqref{u difference tends to 0} that
				\begin{align*}
					&\liminf_{n\rightarrow \infty}
					\E\sup_{t\in[0,T_l\wedge\tau^{-1,n}_{R_0}\wedge\tau^{1,n}_{R_0}]}
					\|u^{-1,n}(t)-u^{1,n}(t)\|_{H^s}\\
					\gtrsim \,  & 
					\liminf_{n\rightarrow \infty}
					\E\sup_{t\in[0,T_l\wedge\tau^{-1,n}_{R_0}\wedge\tau^{1,n}_{R_0}]}
					\|u_{-1,n}(t)-u_{1,n}(t)\|_{H^s}\\
					&-\lim_{n\rightarrow \infty}
					\E\sup_{t\in[0,T_l\wedge\tau^{-1,n}_{R_0}\wedge\tau^{1,n}_{R_0}]}
					\|u_{-1,n}(t)-u^{-1,n}(t)\|_{H^s}\\
					&-\lim_{n\rightarrow \infty}\E\sup_{t\in[0,T_l\wedge\tau^{-1,n}_{R_0}\wedge\tau^{1,n}_{R_0}]}
					\|u_{1,n}(t)-u^{1,n}(t)\|_{H^s}\\
					\gtrsim \,  & \liminf_{n\rightarrow \infty}
					\E\sup_{t\in[0,T_l\wedge\tau^{-1,n}_{R_0}\wedge\tau^{1,n}_{R_0}]}
					\|u_{-1,n}(t)-u_{1,n}(t)\|_{H^s}.
				\end{align*}
			Then it follows from the construction of $u_{m,n}$, Lemmas \ref{lemma ul Hr}, \ref{lemma uh Hr} and \ref{exiting time infty lemma}, and Fatou's lemma that
				\begin{align*}
					&\liminf_{n\rightarrow \infty}
					\E\sup_{t\in[0,T_l\wedge\tau^{-1,n}_{R_0}\wedge\tau^{1,n}_{R_0}]}
					\|u_{-1,n}(t)-u_{1,n}(t)\|_{H^s}\notag\\
					=\,  & \liminf_{n\rightarrow \infty}
					\E\sup_{t\in[0,T_l\wedge\tau^{-1,n}_{R_0}\wedge\tau^{1,n}_{R_0}]}
					\left\|-2n^{-\frac{\delta}{2}-s}\phi\left(\frac{x}{n^{\delta}}\right)\sin(nx)\sin(t)
					+[u_{l,-1,n}(t)-u_{l,1,n}(t)]\right\|_{H^s}\notag\\
					\gtrsim \,  & \liminf_{n\rightarrow \infty}\E\sup_{t\in[0,T_l\wedge\tau^{-1,n}_{R_0}\wedge\tau^{1,n}_{R_0}]}n^{-\frac{\delta}{2}-s}\Big\|\phi\left(\frac{x}{n^{\delta}}\right)\sin(nx)\Big\|_{H^s}|\sin t|-\liminf_{n\rightarrow \infty} n^{\frac{\delta}{2}-1}\notag\\
					\gtrsim \,  &  \sup_{t\in[0,T_l]}|\sin t|,
				\end{align*}
				which is \eqref{sup sin t}. The proof is therefore completed.
			\end{proof}


		\end{document}